\documentclass[10pt,a4paper]{scrartcl}

\usepackage[utf8]{inputenc}
\usepackage{amsbsy}
\usepackage{amsmath}
\usepackage{amssymb}
\usepackage{amsthm}
\usepackage{dsfont}
\usepackage{enumitem}
\usepackage{import}
\usepackage{mathtools}
\usepackage{microtype}
\usepackage{overpic}
\usepackage{xcolor}
\usepackage{bbm}

\usepackage[unicode,pdfborder={0 0 0}]{hyperref}
\usepackage[capitalize,noabbrev]{cleveref}

\mathtoolsset{centercolon}

\newlist{assumptionList}{enumerate}{1}
\setlist[assumptionList]{resume,label={\,C\arabic*.},ref={C\arabic*},format=\textbf,labelindent=0em,leftmargin=*}
\crefname{assumptionListi}{Assumption}{Assumptions}

\crefname{equation}{equation}{equations}

\newcommand{\diff}{\,\mathrm d}
\newcommand{\indicator}{\mathds{1}}
\newcommand{\EE}{\mathds{E}}
\newcommand{\NN}{\mathds{N}}
\newcommand{\PP}{\mathds{P}}
\newcommand{\QQ}{\mathds{Q}}
\newcommand{\RR}{\mathds{R}}

\newcommand{\scE}{\mathcal{E}}
\newcommand{\scF}{\mathcal{F}}

\newcommand{\scI}{\mathcal{I}}
\newcommand{\scL}{\mathcal{L}}
\newcommand{\scS}{\mathcal{S}}
\newcommand{\scW}{\mathcal{W}}

\newcommand{\boldalpha}{\boldsymbol{\alpha}}
\newcommand{\boldh}{\boldsymbol{h}}
\newcommand{\boldw}{\boldsymbol{w}}

\DeclarePairedDelimiter\paren{\lparen}{\rparen}
\DeclarePairedDelimiter\brackets{[}{]}
\DeclarePairedDelimiter\braces{\{}{\}}

\DeclarePairedDelimiter\abs{\lvert}{\rvert}
\DeclarePairedDelimiter\norm{\lVert}{\rVert}

\DeclarePairedDelimiterX{\closedStochasticInterval}[1]{[}{]}{\!\delimsize[#1\delimsize]\!}
\DeclarePairedDelimiterX{\leftOpenStochasticInterval}[1]{]}{]}{\!\delimsize]#1\delimsize]\!}
\DeclarePairedDelimiterX{\rightOpenStochasticInterval}[1]{[}{[}{\!\delimsize[#1\delimsize[\!}

\newcommand{\hittingTimeFromTo}[2]{H^{#1 \,\to\, #2}}

\newcommand{\transpose}{\top}

\newcommand{\Var}{\operatorname{Var}}

\newcommand{\correlationBW}{\rho}
\newcommand{\transformThetaToR}{\mathbbm{r}}

\newcommand{\impactVolatility}{\hat\sigma}
\newcommand{\boundary}{\mathbbm{y}}

\newcommand{\proceedsProcess}{L}

\newcommand{\assetsProcess}{\Theta}
\newcommand{\baseS}{\bar S}

\newcommand{\monotoneStrategies}[1]{\mathcal{A}(#1)}
\newcommand{\reflectionStrategy}[2]{A^\text{refl}(#1,#2)}
\newcommand{\bdryV}{V_\text{bdry}}
\newcommand{\waitV}{V^\scW}
\newcommand{\bothSellV}{V^\scS}
\newcommand{\hardSellV}{V^{\scS_1}}
\newcommand{\easySellV}{V^{\scS_2}}

\newcommand{\lagrangeMultiplier}{m}
\newcommand{\functionalVariation}{\mathrm{D}}

\newtheoremstyle{boldremark}
	{\topsep}   
	{\topsep}   
	{}          
	{}          
	{\bfseries} 
	{.}         
	{.5em}      
	{}          

\newtheorem{theorem}{Theorem}[section]
\newtheorem{proposition}[theorem]{Proposition}
\newtheorem{lemma}[theorem]{Lemma}

\theoremstyle{definition}
\newtheorem{definition}[theorem]{Definition} 
\newtheorem{assumption}[theorem]{Assumption}

\theoremstyle{boldremark}

\newtheorem{remark}[theorem]{Remark}

\numberwithin{equation}{section}

\newcommand{\squeeze}[2][0]{%
  \mbox{$\medmuskip=#1mu\displaystyle#2$}%
}

\title{Optimal Liquidation under Stochastic~Liquidity} 

\author{
	Dirk Becherer\thanks{All authors would like to thank Wolfgang Runggaldier (Co-Editor) and the anonymous two referees and Associate Editor for very helpful comments.}
	, 
	Todor Bilarev\thanks{T.\ Bilarev gratefully acknowledges support by German Science foundation DFG via Berlin Mathematical School BMS and research training group RTG1845 StoA.}
	,
	Peter Frentrup\footnote{Email addresses: becherer,bilarev,frentrup@math.hu-berlin.de}
\\
	Institute of Mathematics, Humboldt-Universität zu Berlin
}

\begin{document}

\maketitle

\begin{abstract}
We solve explicitly a two-dimensional singular control problem of finite fuel type for infinite time horizon. The 
 problem stems from 
the optimal liquidation of an asset position in a financial market with multiplicative and transient price impact. Liquidity is
stochastic in that the volume effect process, which determines the inter-temporal resilience of the market in spirit of \cite{PredoiuShaikhetShreve11},
is taken to be stochastic, being driven by own random noise.
The optimal control is obtained as 
the 
local time of a diffusion process reflected at a non-constant free boundary.
To solve the HJB variational inequality and prove optimality, we need a combination of probabilistic arguments and
calculus of variations methods, involving   
 Laplace transforms
of inverse local times for diffusions reflected at elastic boundaries.

	
\textbf{MSC2010 subject classification}:
	35R35, 49J40, 49L20, 60H30, 60J50, 60J55, 93E20, 91G80
	
\end{abstract}

\section{Introduction}

A typical stochastic optimal control problem in models of illiquid markets is a large trader (the controller) who optimizes her trading strategy such as to balance some trading objective against her adverse price impact, which causes (non-proportional) cost from illiquidity. 
In the majority of literature on price impact models the inter-temporal impact is typically a deterministic function of the strategy of the (single) large trader. 
In reality, we would rather expect some aspects of market liquidity (where \cite{Kyle85} has distinguished resilience, depth and tightness)
 to vary stochastically over time, 
and a sophisticated trader to adapt her optimal strategy accordingly. 
Even for the extensively studied problem of optimal liquidation, there are relatively few recent articles on models in continuous time where the optimal liquidation strategy is adaptive to random changes in liquidity, cf.\
\cite{%
	Almgren12,%
	LorenzSchied13,%
	FruthSchonUrusov17,%
           GraeweHorstSere13,%
	HorstGraewe16}.

We consider a model where temporary market imbalances
involve {\em own stochasticity}.
Price impact is transient, i.e.~it could be persistent but eventually vanishes over time. Moreover, it is non-linear, corresponds to a general shape for the density of the  limit order book (see \cref{rmk: LOB perspective}),  and is multiplicative to ensure positive risky asset prices.
More precisely, our price process $S = (S_t)_{t \ge 0} = \paren{f(Y_t) \baseS_t}_{t \ge 0}$ observed in the market deviates by a factor $f(Y_t)$ from the fundamental price $\baseS_t$ that would prevail in the absence of large traders.
The impact function $f$ is positive and increasing and thus the multiplicative structure ensures that prices stay positive, in contrast to the additive models where a conceptual drawback is that negative asset prices can occur with (small) positive probability.
Our stochastic impact process $Y$ is of a controlled Ornstein-Uhlenbeck (OU) type, namely it is driven by a Brownian motion and the large trader's holdings in the risky asset (see eq.~\eqref{def: impact dynamics} below). The mean-reversion of $Y$ models the transience of impact.
Analogously to  \cite{AlfonsiFruthSchied10,PredoiuShaikhetShreve11}, the impact function $f$ can be linked to the shape of a limit order book (LOB) and $Y$ may be understood as a volume effect process describing the (temporal) imbalance in the LOB, see \cref{rmk: LOB perspective}. The additional noise in $Y$ gives a stochastic LOB, or it can be seen as the accumulated effect from other
non-strategic large traders, see \cref{rmk:stochasticity}.

For our multiplicative model with 
transient impact, we take the fundamental price $\baseS$  to be an exponential Brownian motion and permit for 
non-zero 
correlation with the stochastic volume effect process $Y$.
In this setup, we study the optimal liquidation problem for infinite time horizon as a singular stochastic control problem of finite fuel type and construct its explicit solution. 
Our main result, \cref{thm:optimal strategy}, gives the optimal strategy as the local time 
process of a diffusion reflected obliquely at a curved free boundary in $\RR^2$, the state space being the impact level and the holdings in the risky asset. 
The stochasticity of our optimal strategy arises from its adaptivity to the transient component of the price dynamics and is of local time type. In contrast to the models with additive price impact where the martingale part of the fundamental price is irrelevant for a risk-neutral trader, here the volatility of $\baseS$ is relevant, cf.~\cref{rmk:vol of baseS}.

We solve the singular control problem by explicitly constructing the value function as a classical solution of the HJB variational inequality. Our verification arguments differ from a more common 
approach (outlined in \cref{rmk: calculus of variations approach crucial}) since we were not able  to verify the optimality more directly,
due to the technical complications arising from the implicit nature of the eigenfunctions of the infinitesimal generator for the OU process (see \cref{rmk:conj about par cyl funs}). 
In contrast, we first restrict the set of optimization strategies to those described by diffusions reflected at monotone boundaries, and optimize over the set of possible boundaries. 
To be able to apply methods from calculus of variations, we derive an explicit formula (eq.~\eqref{eq: Laplace transform of inv loc time}) for the Laplace transform of the inverse local times of diffusions reflected at elastic boundaries, i.e.~boundaries which vary with the local time that the reflected process has spent at the boundary, and employ a change of coordinates.
By solving the calculus of variations problem, we construct the candidate optimal free boundary and, moreover, show (one-sided) local optimality in the sense of \cref{thm: one-sided optimal y}. 
The latter is crucial for our verification of optimality.

The paper is organized as follows. 
\cref{sec:heuristics}  states 
the solution 
for the singular stochastic control problem posed in \Cref{sec:problem setting}, 
and outlines the
general 
course of arguments  to come.
In \cref{sec: boundary}, a calculus of variations problem is posed,  by restricting to strategies given by diffusions reflected at smooth  boundaries.
The  free boundary is thereby constructed in \cref{sec:free boundary}.
By solving the HJB variational inequality \eqref{eq:Variational}, we prove optimality and derive the value function and the optimal control in \cref{sec:value fct}.

\section{The model and the optimal control problem}
\label{sec:problem setting}
We consider a filtered probability space $(\Omega, \scF, (\scF_t)_{t\ge0}, \PP)$ with two correlated Brownian motions $W$ and $B$ with correlation coefficient $\correlationBW \in [-1, 1]$, such that
\begin{equation*}
	[W, B]_t = \correlationBW t \,, \quad  t\geq 0.
\end{equation*}
for the quadratic co-variation of $W$ and $B$. The filtration $(\scF_t)_{t\ge0}$ is assumed to satisfy 
 the usual conditions of completeness and right continuity,  so we can take c\`adl\`ag versions for semimartingales. For notions from stochastic analysis we refer to \cite{JacodShiryaev2003_book}.

We consider a market with a risky asset, in addition to the riskless nu\-m\'eraire asset whose (discounted) price is constant at $1$.
The large investor holds $\assetsProcess_t \ge 0$ shares of the risky asset at time $t$. 
She may liquidate her initial position of $\assetsProcess_{0-}$ shares by trading according to
\begin{equation*}
	\assetsProcess_t := \assetsProcess_{0-} - A_t \,,
\end{equation*}
where $A$ is a predictable, c\`adl\`ag, monotone process, describing the cumulative number of assets sold up to time $t$. 
We define the set of admissible strategies
as
\begin{equation*}
	\begin{split}
		\monotoneStrategies{\assetsProcess_{0-}} := \{ A : {}&A \text{ non-decreasing, c\`adl\`ag, predictable,}
	\\		&\text{with } 0=:A_{0-} \le A_t \le \assetsProcess_{0-} \}.
	\end{split}
\end{equation*}
The unaffected fundamental price $\baseS=(\baseS_t)_{t\ge 0}$   of the risky asset evolves according to
\begin{equation} \label{def:unaffected-price}
	\diff\baseS_t = \mu \baseS_t \diff t + \sigma \baseS_t \diff W_t \, , \quad \baseS_0\in (0,\infty), \text{ with $\sigma>0$, $\mu\in \RR$},
\end{equation} 
as a geometric Brownian motion, in the absence of perturbations by large investor trading. 
By  trading, however, the large investor has 
market impact on the actual price
\begin{equation} \label{Sprocdef}
	S_t := f(Y_t) \baseS_t \,,
\end{equation}
of the risky asset through some impact process $Y$, by an increasing positive smooth function $f>0$
with $f(0) = 1$. 
The process $Y$ can be interpreted as a volume effect process, representing the 
 transient volume displacement by large trades in a limit order book (LOB) whose
shape corresponds to the price impact function $f$, 
see \cref{rmk: LOB perspective}.
For $\impactVolatility > 0$ the effect from perturbations $\impactVolatility \diff B_t - \!\diff A_t$  on the  process 
\begin{equation} \label{def: impact dynamics}
	\diff Y_t = -\beta Y_t \diff t + \impactVolatility \diff B_t - \diff A_t \,, \quad Y_{0-} = y,
\end{equation}
is transient over time, in that $Y$ is mean reverting towards zero with mean reversion rate $\beta > 0$.
Existence and uniqueness of a strong solution to \eqref{def: impact dynamics} are guaranteed for instance by \cite[Thm.~4.1]{PangTalrejaWhitt2007}.
Sometimes we shall write $Y^{y,A}$ to stress the dependence of $Y$ on its initial state $y$ and the strategy $A$.
The dynamics of $Y$ are of Ornstein-Uhlenbeck type, driven by $\impactVolatility \diff B - \!\diff A$.
The mean-reversion property of the OU process has the financial interpretation that in the absence of activity from the large trader, the impact lessens since $Y$ reverts back to the neutral state zero and hence the price recovers to the fundamental price $\baseS$, thus modeling the transient component of the impact (in absolute terms).

For $\gamma\ge 0$,  the $\gamma$-discounted proceeds up to time $t$ from a liquidation strategy $A$ are
\begin{equation}\label{def: proceeds}
	\proceedsProcess_t(y;A) := \int_0^t \! e^{-\gamma u} f(Y_u) \baseS_u \diff A^c_u + \sum_{\substack{0 \le u \le t \\ \Delta A_u \ne 0}} e^{-\gamma u} \baseS_u \int_0^{\Delta A_u} \!\! f(Y_{u-} - x) \diff x,
\end{equation}
for $t \ge 0$, where $A_t = A^c_t + \sum_{u\le t}\Delta A_u$ is the (pathwise) decomposition of $A$ into its continuous and pure-jump part, 
and $Y=Y^{y,A}$ solves \eqref{def: impact dynamics}.
Jump terms in \eqref{def: proceeds} can be justified from a LOB perspective (cf.\ \cref{rmk: LOB perspective} below) or by stability results, 
see \cite[Sect.~5.3]{BechererBilarevFrentrup-model-properties} for details.
In particular, if $A^n \to A$ converges in the Skorokhod $M_1$ topology in probability for, e.g., continuous strategies $A^n$ and possibly non-continuous $A$, then the above definition ensures that $\proceedsProcess(y;A^n) \to \proceedsProcess(y;A)$ in $M_1$ in probability.

As $\proceedsProcess$ is an increasing process, the limit $\proceedsProcess_\infty:=\lim_{t\to\infty} \proceedsProcess_t$ exists. 
The large trader's objective is to maximize expected (discounted) proceeds over an infinite time horizon,
\begin{equation}\label{eq:def of value fn v}
	\max_{A \in \monotoneStrategies{\assetsProcess_{0-}} } \EE[\proceedsProcess_\infty(y;A)]\,\quad \text{ with}\quad v(y,\theta):= \sup_{A \in \monotoneStrategies{\theta} } \EE[\proceedsProcess_\infty(y;A)],
\end{equation}
where $v(y,\theta)$ denotes the value function for $y\in\RR$ and $\theta \in [0,\infty)$.

\begin{remark}\label{rmk:v is monotone}
The value function $v$ is increasing in $y$ and $\theta$. 
Indeed, monotonicity in $\theta$ follows from $\monotoneStrategies{\theta_1}\subset \monotoneStrategies{\theta_2}$ for $\theta_1 \leq \theta_2$. For monotonicity in $y$, note that for  $y_1 \leq y_2$ and any strategy $A\in \monotoneStrategies{\theta} $ one has $Y^{y_1,A}_t \leq Y^{y_2,A}_t$ for all $t$,
implying
$\proceedsProcess_t(y_1;A) \leq \proceedsProcess_t(y_2;A)$.
\end{remark}
\noindent
For the rest of the paper,
the function $f$ and scalars $\beta,\mu,\gamma,\sigma,\correlationBW, \impactVolatility$ satisfy
\begin{assumption}\label{ass:modelpara}
~
\begin{assumptionList}
	\item \label{ass:decreasing prices}
		We have $\delta:=\gamma-\mu >0 $, that means the drift coefficient $-\delta \baseS\,$ for the $\gamma$-discounted fundamental price $e^{-\gamma t}\baseS_t$ is negative.

	\item \label{ass:uniqueness of y0}
		The impact function $f\in C^3(\RR)$ satisfies 
		$f,f' > 0$ and $(f'/f)' < (\Phi'/\Phi)'$, where 
		\begin{equation}\label{eq:Phi via H}
			\Phi(x) :=\Phi_\delta(x) := H_{-\delta/\beta} \paren[\Big]{ \paren{ \sigma \correlationBW \impactVolatility - \beta x} / \paren[\big]{\sqrt{\beta} \impactVolatility } },
		\end{equation}
		with Hermite function $H_\nu$ (cf.~\cite[Sect.~10.2]{Lebedev1972special}) and $\sigma,\impactVolatility,\beta > 0$ and $\correlationBW \in [-1,1]$.
		
	\item \label{ass:uniqueness of yinf}
		The impact function $f$ furthermore satisfies $(f'/f)' < (\Phi''/\Phi')'$.

	\item \label{ass:lambda is bounded} 
       The function $\lambda(y):= f'(y)/f(y)$, $y\in \RR$, is bounded, i.e.\ there exists $\lambda_{\max}\in (0,\infty)$ such that  $0 < \lambda(y) \le  \lambda_{\max}$ for all $y \in \RR$. 

	\item \label{ass:decreasing k}
		The function $k(y) := \frac{\impactVolatility^2}{2} \frac{f''(y)}{f(y)} - (\beta + \delta) + (\sigma \correlationBW \impactVolatility - \beta y) \frac{f'(y)}{f(y)}$ is strictly decreasing.
		
	\item \label{ass:existence of y0 and yinf}
		There exist $y_0$ and $y_\infty\in \RR$ such that $(f'/f)(y_0) = (\Phi'/\Phi)(y_0)$ and $(f'/f)(y_\infty) = (\Phi''/\Phi')(y_\infty)$ holds.
	
\end{assumptionList}
\end{assumption}

\cref{ass:modelpara} is satisfied by e.g.\ $f(y) = \exp(\lambda y)$ with $\lambda \in (0, \infty)$, cf.~\cref{lemma: Turan like inequality for Phi} below. 
See \cite[Section~2.1]{BechererBilarevFrentrup2016-deterministic-liquidation} for the shape of the related multiplicative LOB.
Note that $\Phi$ is (up to a constant factor) the unique positive and increasing 
solution of the ODE $ \frac{\impactVolatility^2}{2} \Phi''(y) + (\sigma \correlationBW \impactVolatility - \beta y)\Phi'(y) - \delta \Phi(y) =0$. 

The overall negative drift in \cref{ass:decreasing prices} ensures that the  optimization problem
on an infinite time horizon has a finite value. 
\Cref{ass:uniqueness of y0,ass:uniqueness of yinf} imply uniqueness of the (boundary) points $y_0$ and $y_\infty$ from \cref{ass:existence of y0 and yinf} which are needed in \cref{lemma: boundary}. 
While \ref{ass:uniqueness of yinf}, uniqueness of $y_\infty$, is not crucial there, it will be needed in \eqref{ineq:negative k in sell region} for the verification.
The bound on $\lambda$ in \cref{ass:lambda is bounded} is used to show some growth condition on the value function in \cref{lemma:V is smooth}, that is required to apply the martingale optimality principle (\cref{prop: supermatringale suffices}).
\Cref{ass:decreasing k} is needed for the verification \cref{lemma:inequality in sell region}.

Let us now comment on the model and its financial interpretation.

\begin{remark}\label{rmk: LOB perspective}
We explain how the price impact function $f$ can be interpreted in terms of a (static) {\em multiplicative}  limit order book (LOB)
and $Y$ can be viewed as a {\em volume effect process} in spirit of \cite{PredoiuShaikhetShreve11},
which in our context relates the {\em relative} price impact to transient imbalances of volume.
To this end, let us recall the connection between price impact function $f$ and the (general) density of a LOB, see \cite[Sect.~2.1]{BechererBilarevFrentrup2016-deterministic-liquidation} for more detail and examples.
For relative price perturbations $r_t:= S_t / \baseS_t$, let $q(r)\diff r$ denote the density of offers available at price $r \baseS_t$.
We call the (signed) measure induced by $q$ the multiplicative  limit order book. 
Its cumulative distribution function is $Q(r):= \int_1^r q(x) \diff x$.
The total volume of assets available for prices in some (interval) range $\{ r \baseS_t : r \in R \}$ with measurable $R \subset (0,\infty)$ is $\int_R q(x) \diff x$.
So, a block sale of size $\Delta A_t > 0$ at time $t$ moves the price from $r_{t-}\baseS_t$ to $r_t \baseS_t$ such that the volume changes according to $Q(r_t) = Q(r_{t-}) - \Delta A_t$, giving (discounted) proceeds $e^{-\gamma t} \baseS_t \int_{r_t}^{r_{t-}} x \diff Q(x)$.
In the terminology of \cite{Kyle85}, $Q(r_t) - Q(r_{t-})$ reflects the \emph{depth} of the LOB for price changes by a factor of $r_t/r_{t-}$.
A change of variables with $Y_t:= Q(r_t)$ and $f:= Q^{-1}$ yields the jump term in \eqref{def: proceeds}.
In this sense, $Y$ denotes the effect from the past and present trades on the volume displacement in the  LOB.
By the drift in \eqref{def: impact dynamics}, 
this effect is  persistent over time but not permanent. Its transient nature relates to the liquidity property that \cite{Kyle85} calls resilience. 
Note that in our model the resilience is stochastic in the sense that the volume effect process $Y$   in \eqref{def: impact dynamics} is, whereas the resilience rate $\beta$ is constant (differently e.g.\ to \cite{HorstGraewe16}).

\end{remark}

\begin{remark}\label{rmk:stochasticity}
Stochasticity may account for variations of transient impact that 
cannot be entirely explained by the single agent's own trading activity, and thus not solely described by deterministic functional modeling. 

(a)	
Most of the
literature on transient  impact considers impact that is a deterministic function of the actions of a single large trader.
We consider here an application problem for an individual large trader, but we do not want to assume that she is the only large trader in the market,  or that she is as an aggregate of all large traders (a possibility mentioned in \cite{Frey1998}).
The additional stochastic noise term $\impactVolatility \diff B_t$  in \eqref{def: impact dynamics} can be understood as the aggregate influence on the impact by other large `noise' traders (acting non-strategically).
Questions on strategic behavior between multiple traders
(like in \cite{SchiedZhang15}) 
are interesting  but  beyond the present paper.

(b)
Note that the volatility and as well  the drift of the (marginal) price process $S= f(Y_t) \baseS_t $ from \eqref{Sprocdef}, at which (additional infinitesimally) small quantities of the risky assets would be traded, are stochastic via the additional stochastic component of  $Y$.
Furthermore, we emphasize that the form of relative price impact function $\Delta \mapsto f(Y_{t-} + \Delta)/f(Y_{t-})$
 can vary with $Y$  in general. In the sense of \cref{rmk: LOB perspective}, this means the general shape of
 the corresponding  LOB can exhibit stochastic variations from the large trader's perspective.  

(c)
Recently, \cite{LehalleNeuman17} suggested to model a signal, which predicts the short-term
evolution of prices, as an Ornstein-Uhlenbeck process that modulates the drift of the
price dynamics.
One can interpret stochasticity of $Y$ as such a signal as follows.
For $\lambda = f'/f$ being constant, the log-price can be written as 
$\log S = (\log \baseS + \lambda Y^{\text{sig}}) + \lambda Y^{\text{trans}, \assetsProcess}$, 
where $Y^{\text{sig}}$ is a mean-reverting signal with $\diff Y^{\text{sig}}_t = -\beta Y^{\text{sig}}_t \diff t + \impactVolatility \diff B_t$ and $Y^{\text{trans}, \assetsProcess}$ is the transient impact from trading with $\diff Y^{\text{trans}, \assetsProcess} _t = -\beta Y^{\text{trans}, \assetsProcess} _t \diff t + \diff \assetsProcess_t$.
From this perspective, the optimal liquidation strategy will be adaptive to the signal and depend on the correlation between the signal and $\log \baseS$, 
see \cref{thm:optimal strategy} and \cref{rmk:vol of baseS}.
\end{remark}

\begin{remark}\label{rmk:mesomodel}
Noting that a bid-ask spread is not modeled explicitly and price impact $f$ (i.e.\ the  LOB shape) is static, 
we consider the model as being at a mesoscopic level for low-frequency problems, rather than for market microstructure effects in high frequency.
At this level and as pointed out in \cite[Rmk.\ 2.2]{AlfonsiKloeckSchied16}, it is sensible to think of  price impact and liquidity costs as being aggregated over various types of orders.  The  LOB from \cref{rmk: LOB perspective} should be interpreted accordingly. 
Note that in this paper we deal with monotone strategies and thus only one (bid) side of the LOB is relevant.
Considering infinite time horizon can be viewed as approximation for a longer horizon with more analytic tractability. 
Concerning the question of comparison with additive models for transient impact,  positivity of asset prices is desirable from a theoretical point of view,  relevant for applications with longer time horizons (as they may occur e.g.\ for large institutional trades, cf.\ e.g.\ \cite{ChanL95}, or for hedging problems with longer maturities), and appears to fit better to common models with multiplicative price evolutions like \eqref{def:unaffected-price}. See \cite[Example~5.4]{BechererBilarevFrentrup2016-deterministic-liquidation} for a more detailed discussion  and further references.
\end{remark}

\section{The optimal strategy and how it will be derived}
\label{sec:heuristics}

This section states the main theorem which describes the solution to the singular stochastic control problem,
  and outlines afterwards the general course of arguments for proving it in the subsequent sections.
To explain ideas, let us first motivate how the variational inequality \eqref{eq:Variational}, being the dynamical programming 
equation for the optimization problem at hand,  is readily suggested by an application of the martingale optimality principle. 
To this end, consider for an admissible strategy $A$ the process 
\begin{equation}\label{eq:process G}
	G_t(y;A) := \proceedsProcess_t(y;A) + e^{-\gamma t} \baseS_t V(Y_t,\assetsProcess_t),
\end{equation}
where $G_{0-}(y;A) = \baseS_0  V(Y_{0-}, \assetsProcess_{0-})$ and $V \in C^{2,1}(\RR\times [0,\infty); [0,\infty))$ is some function.
Suppose $V$ can be chosen such that $G$ is a supermartingale.
Then one  should have 
\begin{align*}
	\baseS_0 V(y,\assetsProcess_{0-}) 
		&= \EE[G_{0-}(y;A)] 
	\\
		&\geq \lim_{T\rightarrow \infty}\EE[\proceedsProcess_{T}(y;A)] + \lim_{T \rightarrow \infty} e^{-\gamma T}\EE[\baseS_T V(Y_T, \assetsProcess_{T})]
	\\
		&= \EE[\proceedsProcess_\infty(y;A)]
\end{align*}
heuristically, provided that the second summand on the right-hand side converges to 0.
Hence, for $V$ being such that $G$ is a supermartingale for every admissible strategy $A$ and a martingale for at least one strategy $A^*$, one can conclude that $V$ is essentially the value function for \eqref{eq:def of value fn v} (modulo the factor $\baseS_0$).
To describe  $V$, one may apply  It\^o's formula to get
\begin{align}\label{eq:Ito for G}
\begin{split}
	\diff G_t = e^{-\gamma t} \baseS_t \Bigl( & \impactVolatility V_y(Y_{t-},\assetsProcess_{t-}) \diff B_t + \sigma V(Y_{t-},\assetsProcess_{t-}) \diff W_t
\\		{}+{}& \paren[\big]{ (\mu-\gamma)V +(\sigma \correlationBW \impactVolatility - \beta Y_{t-}) V_y + \tfrac{\impactVolatility^2}{2} V_{yy} }(Y_{t-},\assetsProcess_{t-}) \diff t
\\		{}+{}& \paren[\big]{ f - V_y - V_\theta }(Y_{t-},\assetsProcess_{t-}) \diff A^c_t
\\		{}+{}& \int_0^{\Delta A_t} \paren[\big]{ f - V_y - V_\theta }(Y_{t-} - x,\assetsProcess_{t-} - x) \diff x \Bigr).
\end{split}
\end{align}
Define, with $\delta= \gamma - \mu$, a differential operator on  $C^{2,0}$ functions $\varphi$ by
\[
	\scL \varphi(y,\theta) := \frac{\impactVolatility^2}{2}\varphi_{yy}(y,\theta) + (\sigma \correlationBW \impactVolatility - \beta y) \varphi_y(y,\theta) - \delta \varphi(y,\theta).
\]
By \cref{eq:Ito for G}, solving the Hamilton-Jacobi-Bellman (HJB) variational inequality
\begin{equation} \label{eq:Variational}
	0 = \max \{ f - V_y - V_\theta \ ,\  \scL V \} 
	\quad 
	\text{with
		$V(y, 0) = 0$, $y\in \RR$,}
\end{equation}
would suffice for $G$  to be a local (super-)martingale.
This suggests the existence of a \textit{sell region}  $\scS$ (action region) where the $\diff A$-integrand $f - V_y - V_\theta$ is zero and it is optimal to trade (i.e.\ sell), and a  \textit{wait   region}  $\scW$  (inaction region) in which the $\diff t$-integrand $\scL V$ is zero and it is optimal 
not to trade.
Assume that the two regions 
\begin{align*}
	\scS &= \{ (y,\theta) \in \RR \times (0,\infty) : \boundary(\theta) < y \}
	\quad\text{and}
	\\
	\scW &= \{ (y,\theta) \in \RR \times (0,\infty) : y < \boundary(\theta) \}
\end{align*}
are separated by a free boundary $\{(y,\theta) : y=\boundary(\theta)\}$.
An optimal strategy, i.e.~a strategy for which $G$ is a martingale, would be described as follows: if $(Y_{0-},\assetsProcess_{0-}) \in \scS$, then perform a block sale of size $\Delta A_0$ such that 
\((Y_0,\assetsProcess_0) = (Y_{0-} - \Delta A_0, \assetsProcess_{0-} - \Delta A_0) \in \partial \scS\,.\)
Thereafter, if $\assetsProcess_0>0$, sell just enough as to keep the process $(Y,\assetsProcess)$ within $\overline \scW$.
In this way, the process $(Y,\assetsProcess)$ should be described by a diffusion process that is reflected at the boundary $\partial\scW \cap \partial \scS$ in direction $(-1,-1)$, i.e.~there is  waiting in the interior and selling at the boundary until all shares are sold, when $(Y,\Theta)$ hits $\{ (y,0) : \boundary(0)\le y \}$. 
For such reflected diffusions, exsistence and uniqueness follow from classical results, see \cref{rmk: existence+uniqueness of strong RSDE solution}, and \cref{lemma: Laplace transform of inv loc time} provides important characteristics which are key to the subsequent construction of the optimal control.
The solution of the optimal liquidation problem is indeed described by the local time process of a diffusion reflected at a boundary which is explicitly given by an ODE. 
This main result is stated as \cref{thm:optimal strategy} below.

In the following sections, we will find the value function for our stochastic control problem by constructing a classical solution of the variational inequality \eqref{eq:Variational}. 
Provided that  the key variational inequalities for the (candidate) solution are satisfied, 
optimality can be verified by typical martingale arguments, see \cref{prop: supermatringale suffices}.
Based on results on reflected diffusions from \cref{lemma: Laplace transform of inv loc time}, we reformulate in \cref{sec: boundary} the optimization problem as a (nonstandard) calculus of variations problem. Its solution, derived in \cref{sec:free boundary}, provides a candidate for the free boundary, separating the regions 
of action and inaction, together with the value function on that boundary. 
Moreover, we show a (one-sided) local optimality property of the derived boundary (cf.~\cref{thm: one-sided optimal y}). 
This will be crucial  in \cref{sec:value fct}  (cf.~proof of \cref{lemma:inequality in sell region}) to verify \eqref{eq:Variational} for the candidate value function, constructed there, in order to finally conclude on p.~\pageref{ProofofMainthm} the proof for

\begin{theorem}
	\label{thm:optimal strategy}
	Let \cref{ass:modelpara} be satisfied. 
	Then the ordinary differential equation
	\begin{equation}\label{eq:optimal boundary ode}
		\boundary'(\theta) = \paren[\bigg]{ \frac{ \paren{ (\Phi')^2 - \Phi\Phi''} (f'\Phi' - f\Phi'') / \Phi }{ \paren{ \Phi\Phi'' - (\Phi')^2 } f'' + \paren{ \Phi'\Phi''  - \Phi\Phi''' } f' + \paren{ \Phi'\Phi''' - (\Phi'')^2 } f }}\!\paren[\big]{ \boundary(\theta) }
	\end{equation}
	with initial condition $\boundary(0) = y_0$ admits a unique solution $\boundary:[0,\infty) \rightarrow \RR$, that is strictly decreasing and maps $[0, \infty)$ bijectively to $(y_\infty, y_0]$, for $y_0$ and $y_\infty$ from \cref{ass:existence of y0 and yinf}.
		
	The boundary function $\boundary$ characterizes the solution of problem~\eqref{eq:def of value fn v} as $A^* =(\Delta + K)\indicator_{\closedStochasticInterval{0, \tau}}$, where $\Delta := \assetsProcess_{0-}\indicator_{\{Y_{0-}\ge y_0 + \assetsProcess_{0-}\}} +  \tilde{\Delta}\indicator_{\{Y_{0-}< y_0 + \assetsProcess_{0-}, \tilde{\Delta}\ge 0\}}$ with $\tilde{\Delta} \leq \assetsProcess_{0-}$ satisfying $Y_{0-} - \tilde{\Delta} = \boundary(\assetsProcess_{0-} - \tilde{\Delta})$, and where
	 $(Y, K)$ is the unique continuous adapted process on $\closedStochasticInterval{0,\tau}$ with non-decreasing $K$  which solves  the $\boundary$-reflected SDE 
		\begin{align*}
			Y_t &\le \boundary(\assetsProcess_{0-} - \Delta - K_t)\,,
		\\
			\diff  Y_t &= - \beta Y_t \diff t + \impactVolatility \diff B_t  - \diff K_t \,,
		\\
			\diff K_t &= \indicator_{\braces{  Y_t = \boundary(\assetsProcess_{0-} - \Delta - K_t) }} \diff K_t \,,
		\end{align*}
		starting in $(Y_{0-} - \Delta, 0)$,  for time to liquidation $\tau := \inf\{t\geq 0 : K_t = \assetsProcess_{0-} - \Delta \}$. \\ Moreover, $\tau$ has finite moments.
\end{theorem}

\begin{remark}\label{rmk: Opt strategy explained}
	The optimal control $A^*$ acts as follows: 
	1) If $Y_{0-} \ge y_0 + \assetsProcess_{0-}$, sell everything immediately at time $0$ and stop trading;
	2) Otherwise, if $(\assetsProcess_{0-},Y_{0-})$ is such that $\boundary(\assetsProcess_{0-}) < Y_{0-} < y_0 + \assetsProcess_{0-}$, perform an initial block trade of size $A^*_0 := \Delta > 0$ so that $Y_0 = Y_{0-} - \Delta$ is on the boundary $Y_0 = \boundary(\assetsProcess_0)$.
	Now being in the wait region $\overline{\scW}$, sell as much as to keep with the least effort the state process $(Y,\assetsProcess)$ in $\overline \scW$ until all assets are liquidated at time $\tau$
	(cf.~\cref{fig:impact-and-optimal-strategy-over-time}: waiting e.g.\ at times $t \in [25, 34]$ since then impact $Y_t$ is less than $\boundary(\assetsProcess_t)$).
\end{remark}

The inverse local time $\tau_\ell:= \inf\braces{ t>0 : K_t > \ell }$ is simply how long it takes to liquidate $\ell$ assets (after an initial block sale).
For $\tau > 0$ (case 2 in \cref{rmk: Opt strategy explained}) its Laplace transform is
	\begin{equation}\label{LaplacetrafoTimeToLiquid}
		\EE\brackets[\big]{ e^{-\alpha \tau_\ell} } 
			= \frac{\Phi_{\alpha}\paren{Y_{0} }}{\Phi_{\alpha}\paren{\boundary(\assetsProcess_{0})}} \exp\paren[\bigg]{ \int_0^\ell \paren[\big]{ \boundary'( \assetsProcess_{0} - x) + 1 } \frac{ \Phi_{\alpha}'\paren{ \boundary(\assetsProcess_{0} - x) } }{ \Phi_{\alpha}\paren{ \boundary(\assetsProcess_{0} - x) } } \diff x }
	\end{equation}
	for $\alpha > 0$ and $0\leq \ell\leq \assetsProcess_0=  \assetsProcess_{0-} - \Delta$, as it will be shown in the proof of \cref{thm:optimal strategy}.
	Using analyticity of $\Phi_\alpha$ w.r.t.\ the parameter $\alpha$, one easily gets that $\tau_\ell$ has finite moments.
	Moreover, the Laplace transform \eqref{LaplacetrafoTimeToLiquid} gives access to the distribution of the time to liquidation $\tau$ by efficient numerical inversion, as in e.g.\ \cite{AbateWhitt1995}.

\begin{figure}[ht]
	\centering
	\begin{overpic}[width=0.75\linewidth]{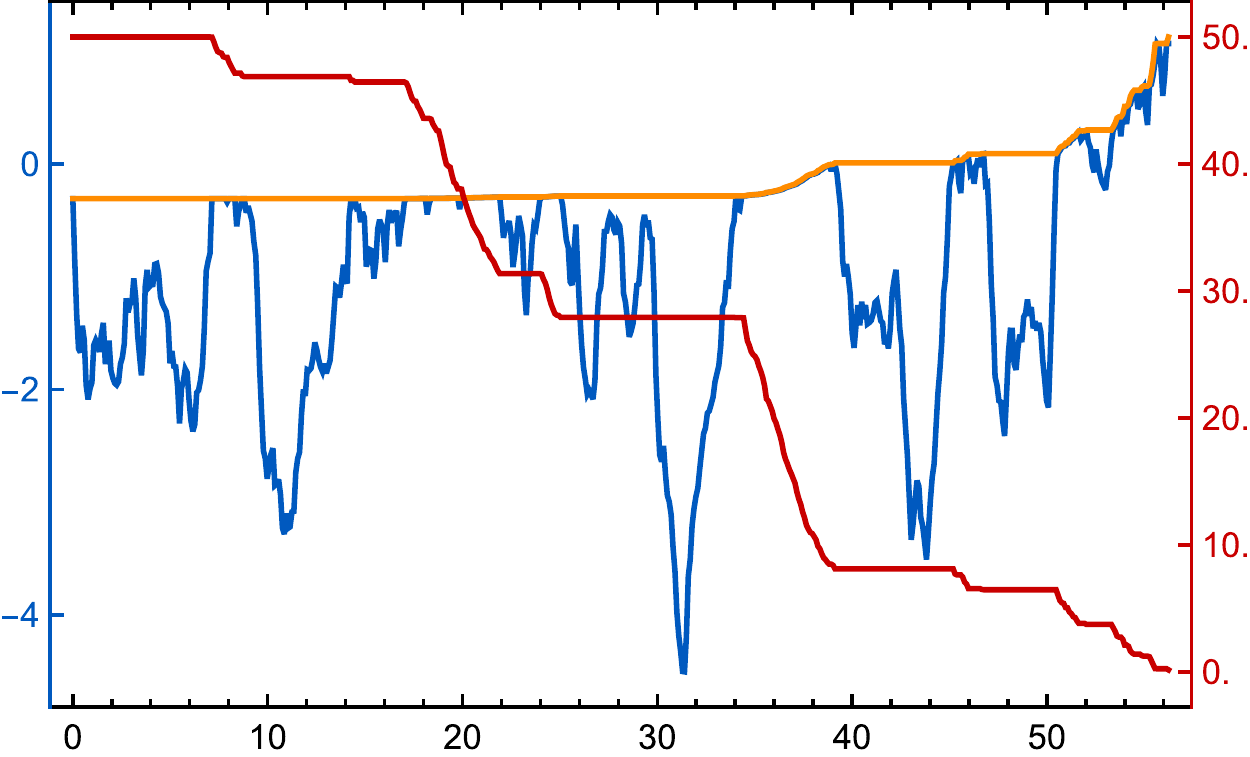}
		\put(95,0.5){$t$}
		\put(1,32){\llap{$Y_t$}}
		\put(100,32){\rlap{$\assetsProcess_t$}}
	\end{overpic}
	\caption{Sample path of impact $Y_t$ (blue), asset position $\assetsProcess_t$ (red, decreasing) and reflecting boundary $\boundary(\assetsProcess_t)$ (orange, increasing) for optimally liquidating $\assetsProcess_0=50$ assets (after an initial block trade $\Delta$), with $\delta=0.1$, $\beta=1$, $\correlationBW=0$, $\impactVolatility=1$ and $f(\cdot)=\exp(\cdot)$.}
	\label{fig:impact-and-optimal-strategy-over-time}
\end{figure}

\begin{remark}\label{rmk:vol of baseS}
The optimizer depends on the volatility of the fundamental price.
	If correlation $\correlationBW$ is not zero, the optimal strategy and the shape of the free boundary  do depend
on the volatility $\sigma$ of the fundamental price process.
This is a notable difference to many additive impact models, 
where the optimal liquidation strategy does not depend on the martingale 
part of the fundamental price process, cf.\ e.g.\ \cite[Sect.\ 2.2]{LorenzSchied13}. 
	To stress the dependence on $\correlationBW$, we write $\Phi^\correlationBW$ for $\Phi$ in \eqref{eq:Phi via H}, denote by $F^\correlationBW$ the right-hand side of \eqref{eq:optimal boundary ode} and by $y^\correlationBW_0$ the root of $f'/f-(\Phi^\correlationBW)'/\Phi^\correlationBW$.
	So the solution $\boundary^\correlationBW$ of the ODE $(\boundary^\correlationBW)'(\theta) = F^\correlationBW\paren[\big]{\boundary^\correlationBW(\theta)}$ with $\boundary^\correlationBW(0) = y^\correlationBW_0$ is the optimal boundary function from \cref{thm:optimal strategy}.
	In the special case of constant $\lambda$, i.e.\ $f(y) = e^{\lambda y}$, we have $F^\correlationBW(y) = F^0(y - \sigma \correlationBW \impactVolatility / \beta)$ since $\Phi^\correlationBW(y) = \Phi^0(y - \sigma \correlationBW \impactVolatility / \beta)$, and thus 
$\boundary^\correlationBW(\theta) = \boundary^0(\theta) + \sigma \correlationBW \impactVolatility / \beta$.
	For general $f$, investigating $y_0$ and $y_\infty$ from \cref{ass:existence of y0 and yinf} still reveals a similar displacement of the boundary.
	Thus, when impact and fundamental price are positively correlated ($\correlationBW > 0$), it is optimal to trade slower if fundamental price volatility is larger, since the wait region increases.
\end{remark}

\begin{remark}
	The optimal liquidation problem with deterministic impact dynamics ($\impactVolatility = 0$) is solved in \cite[Thm.~3.4]{BechererBilarevFrentrup2016-deterministic-liquidation} and characterized by an optimal boundary function $\boundary^0$.
	\Cref{ass:modelpara} implies the model assumptions \cite[Assumption~3.2]{BechererBilarevFrentrup2016-deterministic-liquidation} of that theorem. 
Using the asymptotic expansion \cite[eq.~(10.6.6)]{Lebedev1972special} of Hermite functions, straightforward calculations show uniform convergence $\norm{\boundary^{\impactVolatility} - \boundary^0}_\infty \to 0$ of the boundaries as $\impactVolatility \searrow 0$, for $\boundary^{\impactVolatility}$ solving the ODE \eqref{eq:optimal boundary ode}.
\end{remark}

\section{Reformulation as a calculus of variations problem}\label{sec: boundary}

In this section we will recast the free boundary problem of the variational inequa\-lity~\eqref{eq:Variational} as a (nonstandard, at first) 
calculus of variations problem.
To sketch the idea, suppose that the large trader has to liquidate $\assetsProcess_{0}\geq 0$ shares and that $(Y_{0}, \assetsProcess_{0})$ is already on the free boundary between sell and wait regions (after an initial jump or waiting).
Let $\boundary: [0,\assetsProcess_{0}] \to \RR$ be a $C^1$ function with $\boundary(\assetsProcess_0) = Y_0$ and $\boundary' < 0$
(we expect the optimal boundary to be such).
To find the optimal boundary curve $\boundary$, we will optimize expected proceeds over the set of $\boundary$-reflected strategies $A := \reflectionStrategy{\boundary}{\assetsProcess_0}$ from
\begin{definition}
	Let $(Y, A)$ be the (unique) pair of continuous adapted processes with non-decreasing $A$ such that $Y_t \le \boundary(\assetsProcess_0 - A_t)$ and
	\begin{align*}
		\diff  Y_t &= - \beta Y_t \diff t + \impactVolatility\diff B_t  - \diff A_t \,, & Y_0 &= \boundary(\assetsProcess_0)\,, 
	\\
		\diff A_t &= \indicator_{\braces{  Y_t = \boundary(\assetsProcess_0 - A_t) }} \diff A_t \,, &  A_0 &= 0 \,,
	\end{align*}
	on $\closedStochasticInterval{0, \tau}$ for $\tau := \inf\{t\geq 0 : A_t = \assetsProcess_0 \}$.
	We call $\reflectionStrategy{\boundary}{\assetsProcess_0} := A$ a \emph{$\boundary$-reflected strategy}.
\end{definition}
\begin{remark} \label{rmk: existence+uniqueness of strong RSDE solution}
	Existence and uniqueness of a strong solution $(Y, A)$ follows from (a careful extension of) classical results, cf.~\cite{DupuisIshii93}, by considering the pair $(Y,A)$ as a (degenerate) diffusion in $\RR^2$ with oblique direction of reflection $(-1, +1)$ at a smooth boundary.
	Considered as a one-dimensional diffusion, the process $Y$ is reflected at a boundary that moves with its local time $A$. 
	In this sense, we call the reflection \emph{elastic}.
\end{remark}

\noindent
Viewing $Y$ as a diffusion with reflection at $\boundary$, we can rewrite expected proceeds from $A$ as a deterministic functional of $\boundary$, see \eqref{eq: expected proceeds with y and Phi} below, whose maximizer should describe the optimal strategy.
For this step we rely crucially on a representation for the Laplace transform of the inverse local time of reflected diffusions from \cref{lemma: Laplace transform of inv loc time}.
Since the integrand of \eqref{eq: expected proceeds with y and Phi} depends on the whole path $\boundary$, a reparametrization is necessary to obtain a tractable calculus of variations problem \eqref{problem: maximize with w}~--~\eqref{problem: maximize with w, subsidiary condition}.

Let $\tau_{\assetsProcess_{0}}$ be the stopping time when $A = \assetsProcess_{0}$.
For the continuous $\boundary$-reflected strategy $A$ with proceeds $L := L(\boundary(\assetsProcess_0);A)$, we have by \cite[Theorem~57]{DellacherieMeyer82bookB} for any $T \in [0,\infty)$,
\begin{align*}
	\EE[ \proceedsProcess_T ] 
		&= \EE\brackets[\Big]{ \int_0^{\tau_{\assetsProcess_{0}} \wedge T} f(Y_t) e^{-\delta t} \scE(\sigma W)_t \diff A_t }
\\		&
		= \EE\brackets[\Big]{ \scE(\sigma W)_T \int_0^{\tau_{\assetsProcess_{0}} \wedge T} f(Y_t) e^{-\delta t} \diff A_t }.
\end{align*}
For fixed $T$, let $\QQ$ be the measure given by $\!\diff \QQ / \!\diff \PP = \scE(\sigma W)_T$ on $\scF_T$. Then
\begin{equation} \label{eq: bdry proceeds other measure}
	\EE[\proceedsProcess_T] = \EE^{ \QQ}\brackets[\Big]{ \int_0^{\tau_{\assetsProcess_{0}} \wedge T} f(Y_t) e^{-\delta t} \diff A_t }.
\end{equation}
Girsanov's theorem gives that the process $\widetilde{B}_t := B_t - [B, \sigma W]_t = B_t - \sigma \correlationBW t$ is a Brownian motion under $\QQ$.
Therefore, we have under $\QQ$
\begin{equation*}
	\diff Y_t = (\sigma \correlationBW \impactVolatility - \beta Y_t) \diff t + \impactVolatility\diff \widetilde{B}_t - \diff A_t \,,
\end{equation*}
i.e.~the impact process $Y$ is a (reflected) Ornstein-Uhlenbeck process with shifted (non-zero) mean reversion level, and $A$ is its local time on the boundary.
We cannot directly pass to the limit  $T\rightarrow \infty$ in \eqref{eq: bdry proceeds other measure} because the measure change $\QQ$ depends on $T$.
However, note that the right-hand side of \eqref{eq: bdry proceeds other measure} depends only on the law of the reflected diffusion $(Y, A)$ under the measure $\QQ$.
That is why we consider the reflected diffusion $(X,A^X)$ with the following dynamics under $\PP$: for $g(a):= \boundary(\assetsProcess_0 - a)$ let
\begin{align} 
	\label{def: shifted reflected OU}
	\diff X_t = {}&(\sigma \correlationBW \impactVolatility - \beta X_t) \diff t + \impactVolatility \diff B_t - \diff A^X_t, & X_0 &= g(0) \,,
\\	\label{def: shifted reflected OU local time}
	\diff A^X_t = {}&\indicator_{\{X_t = g(A^X_t)\}} \diff A^X_t, & A^X_0 &= 0 \,,
\\	\label{def: shifted reflected OU inv local time}
	\tau^X_\ell := {}&\inf \{ t > 0 : A^X_t > \ell \text{ or } A^X_t = \assetsProcess_0 \} \,,
\end{align}
such that in addition $X_t \le g(A^X_t)$, on $\closedStochasticInterval{0, \tau^X_{\assetsProcess_0}}$.
Existence and uniqueness of a strong solution $(X, A^X)$ until $\tau^X_{\assetsProcess_0}$ follows as in \cref{rmk: existence+uniqueness of strong RSDE solution}.

Now, by \eqref{eq: bdry proceeds other measure} we have 
\(
	\EE[\proceedsProcess_T] = \EE\brackets{ \int_0^{\tau^X_{\assetsProcess_{0}} \wedge T} f(X_t) e^{-\delta t} \diff A^X_t },
\)
which gives for $T \to \infty$ by monotone convergence on both sides
\begin{align}
	\EE[\proceedsProcess_\infty] &= \EE\brackets[\Big]{ \int_0^{\tau^X_{\assetsProcess_{0}}} f(X_t) e^{-\delta t} \diff A^X_t }\nonumber
		= \EE\brackets[\Big]{ \int_0^{\tau^X_{\assetsProcess_{0}}} f\paren[\big]{g(A^X_t)} e^{-\delta t} \diff A^X_t }
\\		&= \EE\brackets[\Big]{ \int_0^{\assetsProcess_{0}} f\paren[\big]{g(\ell)} e^{-\delta \tau^X_\ell} \diff \ell }
		= \int_0^{\assetsProcess_{0}} f\paren[\big]{g(\ell)} \EE\brackets[\big]{ e^{-\delta \tau^X_\ell} } \diff \ell
		\label{eq: expected proceeds with Laplace transform of inverse local time}
	\,,
\end{align}
using \eqref{def: shifted reflected OU local time}.
To express the latter as a functional of the free boundary only, we need
\begin{theorem} \label{lemma: Laplace transform of inv loc time}
	The Laplace transform of $\tau^X_\ell$ from \eqref{def: shifted reflected OU}--\eqref{def: shifted reflected OU inv local time} for $\assetsProcess_0 = \theta$ is
	\begin{equation} \label{eq: Laplace transform of inv loc time}
		\EE\brackets[\big]{ e^{-\delta \tau^X_\ell} } = \exp\paren[\Big]{\int_{\theta-\ell}^\theta \paren[\big]{\boundary'(x)-1} \frac{\Phi'_\delta(\boundary(x))}{\Phi_\delta(\boundary(x))} \diff x}
		\quad\text{for $\ell < \theta$}.
	\end{equation}
\end{theorem}
\begin{proof}
\newcommand{\generalValue}{q}
	We will identify the Laplace transform by calculating the terms in \eqref{eq: expected proceeds with Laplace transform of inverse local time} at first for $f$ being replaced with arbitrary test functions $\varphi$, and then using ideas from calculus of variations.
	To identify
	$\generalValue(y,\theta) := \EE[\int_0^T e^{-\delta t} \varphi(X_t) \diff A^X_t]$ for continuous functions $\varphi: \RR \to [0,\infty)$ with $X_0=y \le \boundary(\theta)$, $\assetsProcess_0 = \theta$ and $T := \tau^X_\theta$, it suffices to construct $\generalValue$ such that
	\[
		M_t:= \int_0^t e^{-\delta u} \varphi(X_u) \diff A^X_u + e^{-\delta t} \generalValue\paren[\big]{ X_t, \theta - A^X_t }
	\]
	is a martingale on $\closedStochasticInterval{0,T}$ 
	with $e^{-\delta t} \generalValue(X_t, \theta - A^X_t) \to 0$ in $L^1$ as $t \to T$.
	Consider the state space $\scI:= \{ (y,\theta) : y < \boundary(\theta) \}$.
	To check the martingale property, assuming that we have $\generalValue \in C^{2,1}(\scI) \cap C^{1,1}(\overline \scI)$, It\^o's formula yields (similarly to \eqref{eq:Ito for G}) that $\generalValue_y + \generalValue_\theta = \varphi$ on $\partial \scI$ and $\scL \generalValue(y,\theta) = 0$ in $\scI$.
	Moreover, for $\generalValue$ increasing in $y$ we have $\generalValue(y,\theta) = \Phi(y) C(\theta)$ with $\Phi = \Phi_\delta$ from \eqref{eq:Phi via H} and some function $C \in C^1$.
	Let $H(\theta):= \generalValue(\boundary(\theta), \theta)$.
	The condition $q_y + q_\theta = \varphi$ leads to
	\begin{align*}
		H'(\theta)
			&= \Phi'(\boundary(\theta)) C(\theta) \boundary'(\theta) + \paren[\big]{\varphi(\boundary(\theta)) - \Phi'(\boundary(\theta)) C(\theta)}
			= A(\theta) H(\theta) + B(\theta) 
	\end{align*}
	where $A(\theta):= \paren[\big]{ \boundary'(\theta) - 1 } \Phi'(\boundary(\theta)) / \Phi(\boundary(\theta))$ and $B(\theta):= \varphi(\boundary(\theta))$.
	Solving this ODE for $H$ gives (since $H(0)=0$)
	\[
		H(\theta) 
			= \int_0^\theta \!\! \varphi\paren[\big]{ \boundary(z) } \exp\paren[\Big]{\int_{z}^\theta \paren[\big]{ \boundary'(x)-1 } \frac{\Phi'_\delta(\boundary(x))}{\Phi_\delta(\boundary(x))} \diff x} \diff z
			,
	\]
	which yields the candidate \(\generalValue(y,\theta) = \Phi(y) H(\theta) / \Phi(\boundary(\theta))\).
	It is straightforward to check $\generalValue \in C^{2,1}(\scI) \cap C^{1,1}(\overline \scI)$ and $\generalValue_y + \generalValue_\theta = \varphi$ on $\partial \scI$, giving that $M$ is a martingale, using boundedness of $\generalValue_y(X, \theta - A^X)$ on $\closedStochasticInterval{0,T}$.
	By monotonicity of $\generalValue$ in $y$, hence $\generalValue(y,\theta) \leq H(\theta)$, we obtain $e^{-\delta t} \generalValue(X_t, \theta - A^X_t) \to 0$ in $L^1$ as $t \to T$ via dominated convergence, so as in \eqref{eq: expected proceeds with Laplace transform of inverse local time} we find
	\begin{equation} \label{eq:laplace transf difference}
		\int_0^\theta \varphi(\boundary(z)) \paren[\bigg]{ \underbrace{ \EE[e^{-\delta \tau^X_{\theta-z}}] - \exp\paren[\Big]{\int_{z}^\theta \paren[\big]{ \boundary'(x) - 1 } \frac{\Phi'_\delta(\boundary(x))}{\Phi_\delta(\boundary(x))} \diff x}}_{=: \Delta(z)}} \diff z = 0.
	\end{equation}
	Note that $z \mapsto \EE[\exp\paren{-\delta \tau^X_{\theta-z}}]$ is left-continuous.
	Hence, if $\Delta(z_1) > 0$ for some $z_1 \in (0,\theta]$, there exists $z_0 < z_1$ such that $\Delta > 0$ on $(z_0, z_1)$.
	Since $\boundary$ is bijective (recall that $\boundary' < 0$), we can find a continuous function $\varphi$ with $\varphi \circ \boundary > 0$ inside $(z_0,z_1)$ and $\varphi \circ \boundary = 0$ outside $(z_0,z_1)$, which yields $\int_0^\theta \varphi(\boundary(z)) \Delta(z) \diff z > 0$, contradicting \eqref{eq:laplace transf difference}.
	Similarly, $\Delta(z_1) < 0$ also leads to a contradiction.
	Therefore $\Delta = 0$ on $(0,\theta]$.
\end{proof}
\begin{remark}
Let us note that \cref{lemma: Laplace transform of inv loc time} generalizes to general (regular) diffusions reflected at increasing boundaries by taking $\Phi_\delta$ to be the increasing non-negative $\delta$-eigenfunction of the generator of the diffusion. Indeed, the proof would not change. 
\end{remark}
Using \cref{lemma: Laplace transform of inv loc time} and  \eqref{eq: expected proceeds with Laplace transform of inverse local time} we derive the following representation for the proceeds from a $\boundary$-reflected strategy in terms of the boundary:
\begin{equation} \label{eq: expected proceeds with y and Phi}
	\EE[\proceedsProcess_\infty] = \int_0^{\assetsProcess_{0}} f\paren[\big]{ g(\ell) } \exp\paren[\bigg]{ -\int_0^\ell \paren[\big]{ g'(a) + 1 } \frac{\Phi'_\delta(g(a))}{\Phi_\delta(g(a))} \diff a} \diff \ell \,.
\end{equation}
Since the $\!\diff\ell$-integrand in \eqref{eq: expected proceeds with y and Phi} depends on the whole path of $g$, classical calculus of variations methods are not directly available.
Since by definition $g(a) = \boundary(\assetsProcess_0 - a)$ we get with $\transformThetaToR(\ell) :=\int_0^\ell \paren{ 1 - \boundary'(x) } \frac{\Phi'(\boundary(x))}{\Phi(\boundary(x))} \diff x$ that
\begin{equation} \label{eq: expected proceeds with y and rho}
	\EE[ \proceedsProcess_\infty ] = e^{-\transformThetaToR(\assetsProcess_{0})} \int_0^{\assetsProcess_{0}} f\paren[\big]{ \boundary(\ell) } e^{\transformThetaToR(\ell)} \diff \ell.
\end{equation}
Since $\Phi',\Phi>0$ and $\boundary' < 0$, the function $\transformThetaToR$ in strictly increasing and thus has an inverse $\transformThetaToR^{-1}$. 
Fixing $R:= \transformThetaToR(\assetsProcess_{0})$ and setting $w(r):= \boundary(\transformThetaToR^{-1}(r))$, we find
\begin{equation*}
	\transformThetaToR^{-1}(r) = \int_0^r \paren[\bigg]{ w'(z) + \frac{\Phi(w(z))}{\Phi'(w(z))} } \diff z.
\end{equation*}
Hence, by the reparametrization $\boundary(\theta) = w(\transformThetaToR(\theta))$, finding a maximizing function $\boundary$ for \eqref{eq: expected proceeds with y and rho} reduces to the problem of finding a function $w$ which maximizes
\begin{align}\label{problem: maximize with w}
	J(w):= \int_0^R f\paren[\big]{w(r)} e^{-(R-r)} \squeeze[1]{ \paren[\bigg]{ w'(r) + \frac{\Phi(w(r))}{\Phi'(w(r))} } \diff r } \quad (= \EE[L_\infty])
\\
\label{problem: maximize with w, subsidiary condition}
\text{subject to the condition }
	K(w):= \int_0^R \squeeze[1]{ \paren[\bigg]{ w'(r) + \frac{\Phi(w(r))}{\Phi'(w(r))} } \diff r = \assetsProcess_{0} }\,.
\end{align}

\section{Solving the calculus of variations problem}
\label{sec:free boundary}

In this section, we solve (locally) the calculus of variations problem of maximizing~\eqref{problem: maximize with w} subject to~\eqref{problem: maximize with w, subsidiary condition} by employing necessary and sufficient conditions on the first and second variation of the functionals involved. 
We obtain the candidate free boundary function $\boundary(\theta)$, see \cref{def: M1 and M2,eq:boundary}, and show its local optimality in \cref{lemma: strict local maximizer w variational problem}. 
We then relate our results on the calculus of variations problem to the initial optimal execution problem in \cref{thm: one-sided optimal y}. This will be crucial later for \cref{sec:value fct} to verify the desired inequality in the sell region, presented in \cref{lemma:inequality in sell region}.

A maximizer $w$ of the isoperimetric problem \eqref{problem: maximize with w} -- \eqref{problem: maximize with w, subsidiary condition} also maximizes $J + \lagrangeMultiplier K$ for some constant $\lagrangeMultiplier := \lagrangeMultiplier(R)$ that is the Lagrange multiplier, cf.~\cite[Theorem~2.12.1]{GelfandFomin00}.
Considering perturbations $w(r) + h(r)$ of $w$ with $h(0) = h(R) = 0$, a necessary condition for an extremum $w$ of a functional $J + \lagrangeMultiplier K$ is that its first variation $\functionalVariation(J+\lagrangeMultiplier K)$ vanishes at $w$, see \cite[Thm.~1.3.2]{GelfandFomin00}.
Integration by parts yields the Euler-Lagrange equation
\begin{equation}\label{eq: Euler-Lagrange}
	0 = F_w - \frac{\diff}{\diff r}F_{w'} + \paren[\Big]{ G_w - \frac{\diff}{\diff r}G_{w'} } \lagrangeMultiplier
	\,,
\end{equation}
with $G(r,w,w') \!:=\! w' + \Phi(w) / \Phi'(w)$ and $F(r,w,w') \!:=\! f(w)e^{-(R-r)} G(r,w,w')$, the integrands of $K$ and $J$, respectively.

Since we assumed to start on the (yet unknown) boundary, one side is fixed, $w(R)=\boundary(\assetsProcess_{0})$.
But the other end $w(0)$ is free.
Thus, integration by parts of $\functionalVariation(J+\lagrangeMultiplier K)$ with perturbations $w(r) + h(r)$ of $w$ where $h(0) \ne 0$ imposes as an additional condition for $\functionalVariation(J+\lagrangeMultiplier K)$ to vanish that
\begin{equation*}
	0 = \paren[\big]{ F_{w'} + \lagrangeMultiplier G_{w'} } \bigr|_{r=0} \,.
\end{equation*}
This \emph{natural boundary condition} (cf.\ \cite[Sect.\ 1.6]{GelfandFomin00}) yields the Lagrange multiplier $\lagrangeMultiplier(R) = -f(y_0) e^{-R}$ for $y_0:= \boundary(0) = w(0)$.
After multiplication with $e^R \Phi'(w)^2$, equation~\eqref{eq: Euler-Lagrange} simplifies to
\begin{equation}\label{eq: simplified Euler Lagrange}
	e^r \Phi(w) \paren[\big]{f'(w)\Phi'(w) - f(w) \Phi''(w) } = f(y_0) \paren[\big]{\Phi'(w)^2 - \Phi(w)\Phi''(w)}\,.
\end{equation}
Inserting $r=0$ gives a condition for $y_0$, namely
\begin{equation*}
	f'(y_0) \Phi(y_0) = f(y_0) \Phi'(y_0).
\end{equation*}
\Cref{ass:existence of y0 and yinf} guarantees existence and \labelcref{ass:uniqueness of y0} uniqueness of $y_0$.
On the other hand, differentiating both sides of \eqref{eq: simplified Euler Lagrange} with respect to $r$ gives the ODE for $w$
\begin{equation}\label{ode: w with e^r}
\begin{split}
	0 &= \paren[\big]{ e^r (f'\Phi' - f\Phi'')\Phi' + e^r ( f''\Phi' - f\Phi''' )\Phi - f(y_0) (\Phi'\Phi'' - \Phi\Phi''')} w'
\\		&\qquad + e^r (f'\Phi' - f\Phi'') \Phi,
\end{split}
\end{equation}
where $f= f\paren{ w(r) }$, $f' = f'\paren{ w(r) }$, $\Phi= \Phi\paren{ w(r) }$, etc.

Both sides in the above equality~\eqref{eq: simplified Euler Lagrange} are negative on the boundary $w(r)$, due to 
\begin{lemma}\label{lemma: Turan like inequality for Phi}
	The positive, increasing eigenfunctions $\Phi = \Phi_\delta$ corresponding to the eigenvalue $\delta>0$ of the generator of an Ornstein-Uhlenbeck process satisfy 
	\[
		\paren[\big]{ \Phi^{(n)}(x) }^2 < \Phi^{(n-1)}(x)\Phi^{(n+1)}(x)
	\]
	for all $x\in\RR$ and $n \in \NN$. In particular, $(\Phi')^2 < \Phi\Phi''$. Moreover, for $n\in \NN$
	\[
		\lim_{x\rightarrow -\infty}\Phi^{(n)}(x)/\Phi^{(n-1)}(x) = 0\quad \text{and}\quad \lim_{x\rightarrow +\infty}\Phi^{(n)}(x)/\Phi^{(n-1)}(x) = +\infty.
	\]
\end{lemma}
\begin{proof}
	Since $H_\nu'(x) = 2\nu H_{\nu-1}(x)$ for complex $\nu$ (see~\cite[eq.~(10.4.4)]{Lebedev1972special}), \cref{eq:Phi via H} implies
	\[
		\Phi^{(n)}_\delta \Phi^{(n+2)}_\delta - \paren[\big]{ \Phi^{(n+1)}_\delta }^2
		= \paren[\big]{ \Phi_{\delta + n\beta} \Phi''_{\delta + n\beta} - (\Phi'_{\delta + n\beta})^2 } \frac{2^{2n}}{\impactVolatility^{2n}\beta^n} \prod_{k=0}^n (\delta + k\beta)^2 \,,
	\]
	so it suffices to prove $(\Phi')^2 < \Phi''\Phi$ for every $\delta,\beta,\sigma,\impactVolatility > 0$ and $\correlationBW \in [-1,1]$ in \eqref{eq:Phi via H}.
	This is equivalent to showing $(H_\nu')^2 < H_\nu''H_\nu$ for every $\nu < 0$.
	Since $\Gamma(-\nu) > 0$ and $H_{\nu}(x) = \Gamma(-\nu)^{-1}\int_0^\infty e^{-t^2 - 2xt} t^{-\nu-1} \diff t$ for $\nu<0$ (cf.~\cite[eq.~(10.5.2)]{Lebedev1972special}), the function $\varphi_x(t):= e^{-t^2 - 2xt} t^{-\nu-1}$ is the density of an absolutely continuous finite measure $\mu$ on $[0,\infty)$.
	For the probability measure $\tilde\PP[A] := \mu([0,\infty))^{-1} \mu(A)$ consider two independent random variables $X,Y \sim \tilde\PP$.
	By \cite[Thm.~6.28]{Klenke2008probability}, we can exchange differentiation and integration (in the integral representation of $H_\nu$ above) to see that 
	\(
		H_\nu''(x)H_\nu(x) - H_\nu'(x)^2 = 4\, \tilde\EE[ X^2 - X Y ].
	\)
	Symmetry gives $2\, \tilde\EE[ X^2 - X Y ] = \tilde\EE[(X-Y)^2] \ge 0$.
	Since $X$ and $Y$ are independent with absolutely continuous distribution, Fubini's theorem yields $\tilde\PP[X=Y] = 0$, so $\tilde\EE[(X-Y)^2] > 0$.
	
	The asymptotic behavior of $\Phi^{(n)}/\Phi^{(n-1)}$ follows from \cite[eq.~(10.6.4)]{Lebedev1972special} in the case $x\rightarrow -\infty$ and from \cite[eq.~(10.6.7)]{Lebedev1972special} in the case $x\rightarrow +\infty$.
\end{proof}

\noindent
Now \eqref{eq: simplified Euler Lagrange} gives a representation of $r$ given $y_0$ and $w$ as
\begin{equation}\label{eq: r via optimal w}
	r = \log \frac{f(y_0)}{\Phi(w)} + \log \frac{\Phi'(w)^2 - \Phi(w)\Phi''(w)}{f'(w)\Phi'(w)-f(w)\Phi''(w)},
\end{equation}
which we can use to simplify the ODE~\eqref{ode: w with e^r} (assuming $w'\ne 0$ everywhere) to
\begin{equation*}
	\frac{1}{w'} = -\frac{\Phi'}{\Phi} + \frac{f\Phi''' - f''\Phi'}{f'\Phi'-f\Phi''} + \frac{\Phi'\Phi'' - \Phi\Phi'''}{(\Phi')^2-\Phi\Phi''},
\end{equation*}
reading the right hand side as a function of $w(r)$. With $\boundary(\theta) = w(\transformThetaToR(\theta))$ and $r:= \transformThetaToR(\theta)$, we get $\boundary'(\theta) = w'(r) \transformThetaToR'(\theta) = w'(r) (1 - \boundary'(\theta)) \Phi'(\boundary(\theta)) / \Phi(\boundary(\theta))$, which simplifies to
\begin{align}
	\boundary'(\theta) 
		&= \frac{\Phi'(\boundary)}{\Phi'(\boundary) + \Phi(\boundary)/w'(r)}
\notag\\
		&= \frac{1}{\Phi} \frac{ \paren{ (\Phi')^2 - \Phi\Phi''} (f'\Phi' - f\Phi'') }{ (\Phi\Phi'' - (\Phi')^2) f'' + (\Phi'\Phi''  - \Phi\Phi''') f' + (\Phi'\Phi''' - (\Phi'')^2 ) f}
\notag\\
		&= \frac{M_2(\boundary(\theta))}{M_1'(\boundary(\theta))}, \label{eq: y' with M1 and M2}
\end{align}
\begin{equation} \label{def: M1 and M2}
	\text{where }\quad  M_1 := \frac{f \Phi' - f'\Phi}{(\Phi')^2 - \Phi\Phi''} \qquad\text{and}\qquad M_2 := \frac{f'\Phi' - f \Phi''}{(\Phi')^2 - \Phi\Phi''}.
\end{equation}
By \eqref{eq: simplified Euler Lagrange} and \cref{lemma: Turan like inequality for Phi} we have $M_2(\boundary(\theta)) > 0$ for any $\theta$.
We get $M_1'(\boundary(\theta)) < 0$ by
\begin{lemma}\label{lemma: M1 monotonicity}
	Under  \cref{ass:uniqueness of y0},
	 $M_1'(y) < 0$ for all $y \in \RR$.
\end{lemma}
\begin{proof}
	Let $G:= \Phi'/\Phi$ and $H:= \Phi''/\Phi'$. 
	We have $G,G',H,H'>0$ and $G<H$ by \cref{lemma: Turan like inequality for Phi}.
	With $\lambda(y) = f'(y)/f(y) > 0$, thus $f''/f = \lambda' + \lambda^2$, we get
	\[
		(G')^2 \Phi M_1' / f = \lambda' G' + (\lambda^2-\lambda H)G' + (G^2-\lambda G)H'.
	\]
	So $M_1'(y) < 0$ if and only if $\lambda'(y) G'(y) < q(\lambda(y))$ where the right-hand side is $q(\lambda):= (H-\lambda) \lambda  G' + (\lambda - G) G H'$.
	The function $q$ is quadratic in $\lambda$ and takes its minimum in
	\[
		\lambda^*:= \frac{HG'+GH'}{2G'}	\qquad \text{ with value } \qquad q(\lambda^*) = \frac{(HG'+GH')^2}{4G'} - G^2H'.
	\]
	Note also, that $G' = (H-G)G$.  
We find that
	\begin{align*}
		4G' \, (\lambda' G' - q(\lambda)) &\le 4 G' \, (\lambda'G' - q(\lambda^*)) 
			< 4G' \, \paren[\big]{(G')^2 - q(\lambda^*)}
	\\		&= 4 (G')^3 - (GH'+G'H)^2 + 4 G' G^2 H'
	\\		&= G^2 \paren[\Big]{ 4G (H-G)^3 - \paren[\big]{H' + (H-G)H}^2 + 4 (H-G)GH' }
	\\		&= -G^2 \paren[\big]{ H' + H^2 + 2G^2 - 3GH }^2 \le 0,
	\end{align*}
 using that  $\lambda'(y) < G'(y)$, $y\in\RR$,  by \cref{ass:uniqueness of y0}. 
	So $M_1'(y) < 0$ for all $y \in \RR$.
\end{proof}

\begin{lemma}\label{lemma: boundary}
	Let $f$ satisfy \cref{ass:uniqueness of y0,ass:uniqueness of yinf,ass:existence of y0 and yinf}.
	Then there exists a unique solution $\theta \mapsto \boundary(\theta)$, $\theta \in [0,\infty)$, of the ODE 
	\begin{equation}\label{eq:boundary}
		\boundary' = M_2(\boundary) / M_1'(\boundary),\quad  \boundary(0) = y_0, 
	\end{equation}
	and $\boundary$	is strictly decreasing to $\lim_{\theta \to \infty} \boundary(\theta) = y_{\infty}$ (with $y_0$ and $y_\infty$ from \cref{ass:existence of y0 and yinf}).
\end{lemma}
\begin{proof}
	Since $M_2/M_1'$ is locally Lipschitz by $f\in C^3(\RR)$, there exists a unique maximal solution $\boundary : [0,\theta_{\max}) \to \RR$ of \eqref{eq:boundary}.
	We have $M_2(\boundary(\theta)) > 0$ and $M_1'<0$ by \cref{lemma: M1 monotonicity}, thus $\boundary' < 0$.
	Assume $\theta_{\max} < \infty$, which implies $\lim_{\theta \to \theta_{\max}} \boundary(\theta) = -\infty$.
	However, note that $\braces{\paren{\theta, \boundary(\theta)} : 0 \le \theta < \theta_{\max}}$ and $[0,\infty) \times \{y_\infty\}$ are trajectories of the two-dimensional autonomous dynamical system induced by the field $(\theta, y) \mapsto \paren{1, M_2(y)/M_1'(y)}$.
	Since trajectories of autonomous dynamical systems cannot cross, and $y_\infty < y_0$ by \cref{lemma: Turan like inequality for Phi}, we must have $y_\infty < \boundary(\theta)$ for all $\theta \in [0,\theta_{\max})$, which contradicts $\theta_{\max} < \infty$.
	
	Moreover, $\boundary^{-1}(y) = \int_{y_0}^y (M_1'/M_2)(x)\diff x$ is finite for every $y\in (y_\infty, y_0]$. Since $\theta_{\max} = \infty$, it follows that $ \boundary(\theta) \to y_\infty$ as $\theta\to \infty$.
\end{proof}

By considering the first variation $\functionalVariation(J + \lagrangeMultiplier K)$, we found a candidate boundary function $\boundary$ in terms of a possible extremum $w : [0,R] \to \RR$ of $J + \lagrangeMultiplier K$.
Calculating the second variation $\functionalVariation^2(J + \lagrangeMultiplier K)$ at $w$, we find that $w$ is indeed a local maximizer.
\begin{lemma} \label{lemma: strict local maximizer w variational problem}
	The functional $\hat J:= J+\lagrangeMultiplier K: C^1([0,R]) \to \RR$ defined by \eqref{problem: maximize with w} -- \eqref{problem: maximize with w, subsidiary condition} with $\lagrangeMultiplier = -f(y_0) e^{-R}$ has a strict local maximizer $w(r) = \boundary(\transformThetaToR^{-1}(r))$, with $\boundary$ solving \eqref{eq:boundary}, in the following sense.
	There exists $\varepsilon > 0$ such that for all perturbations $0 \not\equiv h \in C^1([0,R])$ with endpoints $h(0) = h(R) = 0$ and $\norm{h}_{W^{1,\infty}} = \norm{h}_\infty \vee \norm{h'}_\infty < \varepsilon$ it holds
	\[
		\hat J(w+h) < \hat J(w).
	\]
\end{lemma}
\begin{proof}
For a $C^1$-perturbation $h:[0,R] \to \RR$ of $w$ with $h(0) = h(R) = 0$ we have by \cite[Sect.~5.25, (10) and (11)]{GelfandFomin00}
\begin{equation*}
	\functionalVariation^2(J + \lagrangeMultiplier K)[w;h] = \int_0^R \paren[\big]{P h'(r)^2 + Q h(r)^2} \diff r
\end{equation*}
with $P = P(r,w(r),w'(r))$ and $Q=Q(r,w(r),w'(r))$ given by
\begin{align*}
	P &= \tfrac{1}{2} \paren[\big]{ F_{w'w'} + \lagrangeMultiplier G_{w'w'}} = 0,
\\
	Q &= \frac{1}{2} \paren[\Big]{ F_{ww} + \lagrangeMultiplier G_{ww} - \frac{\diff}{\diff r} \paren[\Big]{ F_{ww'} + \lagrangeMultiplier G_{ww'} } }
\\		
	&= \frac{1}{2} e^{-(R-r)} \paren[\bigg]{ \frac{\Phi}{\Phi'} f'' + 2 \paren[\Big]{ \frac{\Phi}{\Phi'} }' f' + \paren[\Big]{ \frac{\Phi}{\Phi'} }'' f - f' } + \frac{1}{2} \paren[\Big]{ \frac{\Phi}{\Phi'} }'' \lagrangeMultiplier,
\end{align*}
with $f$, $\Phi$ and their derivatives being evaluated at $w(r)$ when no argument is mentioned.
Differentiating \eqref{eq: Euler-Lagrange} with respect to $r$ yields
\begin{align}
	0 &= \frac{\diff}{\diff r} e^{-(R-r)} \paren[\bigg]{ \frac{\Phi}{\Phi'} f'  + \paren[\Big]{ \frac{\Phi}{\Phi'} }' f - f } + \lagrangeMultiplier \, \frac{\diff}{\diff r} \paren[\Big]{ \frac{\Phi}{\Phi'} }'
	\nonumber
\\		&= e^{-(R-r)} \paren[\bigg]{ \frac{\Phi}{\Phi'} f' + \paren[\Big]{ \frac{\Phi}{\Phi'} }' f - f } 
	\nonumber
\\		&\qquad+ e^{-(R-r)} \paren[\bigg]{ \frac{\Phi}{\Phi'} f'' + 2 \paren[\Big]{ \frac{\Phi}{\Phi'} }' f'  + \paren[\Big]{ \frac{\Phi}{\Phi'} }'' f - f' } w' + \paren[\Big]{ \frac{\Phi}{\Phi'} }'' \lagrangeMultiplier\, w' 
	\nonumber
\\		&= e^{-(R-r)} \paren[\bigg]{ \frac{\Phi}{\Phi'} f' + \paren[\Big]{ \frac{\Phi}{\Phi'} }' f - f } + 2 Q w'
	\nonumber
\\	\label{eq:diff Euler Lagrange}
		&= e^{-(R-r)} \frac{\Phi}{(\Phi')^2} \, \paren[\big]{ f' \Phi' - f \Phi'' } + 2 Q w' \,.
\end{align}
By \cref{eq: simplified Euler Lagrange} and \cref{lemma: Turan like inequality for Phi}, the first summand in~\eqref{eq:diff Euler Lagrange} is negative along $w(r)$.
Since $w(r) = \boundary(\transformThetaToR^{-1}(r))$ and $\transformThetaToR^{-1}$ is strictly increasing, we have $w' < 0$ by \cref{lemma: boundary}.
So $Q(r,w(r),w'(r)) < -\kappa < 0$ on $[0,R]$ by \eqref{eq:diff Euler Lagrange} for some constant $\kappa = \kappa_R$, giving that the second variation is negative definite at $w$, i.e.~for $h \not\equiv 0$,
\begin{equation}\label{ineq: negative second variation}
	\functionalVariation^2(J \!+\! \lagrangeMultiplier K)[w;h] 
		\!=\! \int_0^R \! Q(r,w(r),w'(r)) h(r)^2 \diff r 
		\!<\! - \kappa \! \int_0^R \! h(r)^2 \diff r 
		\!<\! 0 \,.
\end{equation}

	To shorten notation, let $\hat F:= F + \lagrangeMultiplier G$, so $\hat J:=J + \lagrangeMultiplier K =  \int_0^R \hat F \diff r$. Unless the arguments are explicitly written, take $\hat F = \hat F(r, w(r), w'(r))$. Taylor's theorem gives
	\(
		\hat J(w+h) - \hat J(w) = \functionalVariation \hat J[w;h] + \functionalVariation^2 \hat J[w;h] + \scE(h)
	\)
	with first variation $\functionalVariation \hat J[w;h] = 0$ by~\eqref{eq: Euler-Lagrange}, second variation $\functionalVariation^2 \hat J[w;h] = \int_0^R Q h^2 \diff r < 0$ by~\eqref{ineq: negative second variation} and remainder 
	\[
		\scE(h) = \int_0^R \paren[\bigg]{ \sum_{\abs{\boldalpha} = 3} \partial^{\boldalpha} \hat F\paren[\big]{r, \boldw + \xi_r \boldh} \frac{\boldh^{\boldalpha}}{\boldalpha !} } \diff r
	\]
	for some $\xi_r \in [0,1]$, with $\boldw = (w(r), w'(r))^\transpose$, $\boldh = (h(r), h'(r))^\transpose$ and multi-index $\boldalpha \in \NN_0^2$, considering $\hat F(r,\cdot)$ as an function on $\RR^2$.
	Since $\hat F$ is affine in $w'$ we get
	\[
		\scE(h) = \int_0^R \! \paren[\Big]{ \frac{1}{6}\hat F_{www}(r, \boldw + \xi_r \boldh) h + \frac{1}{2} \hat F_{ww w'}(r, \boldw + \xi_r \boldh) h'} h^2 \diff r
		=: \int_0^R \!\!\! A h^2 \diff r
	\]
	Note that by compactness of $[0,R]$ we have uniform convergence 
	\[
		\sup_{r\in[0,R]}\sup_{\xi \in [0,1]} \abs[\big]{ A\paren[\big]{h(r), h'(r), w(r), w'(r), \xi, r} } \to 0
	\] 
	as $\norm{h}_{W^{1,\infty}} \to 0$.
	Now choose $\varepsilon > 0$ small enough such that 
	\[
		\abs[\big]{ A\paren[\big]{h(r), h'(r), w(r), w'(r), \xi, r} } < \kappa/2
	\]
	for all $r \in [0,R]$, $\xi \in [0,1]$ and $h$ with $\norm{h}_{W^{1,\infty}} < \varepsilon$.
	Hence, with $h \not\equiv 0$
	\[
		\hat J(w+h) - \hat J(w) = \int_0^R (Q + A) h^2 \diff r < -\frac{\kappa}{2} \int_0^R h^2 \diff r < 0\,. \qedhere
	\]
\end{proof}
Note that the definition $w(r):= \boundary(\transformThetaToR^{-1}(r))$ does not depend on the interval boundary $R$. 
Hence the optimizer $w$ over $[0,R]$ from \cref{lemma: strict local maximizer w variational problem} is optimal for all $R > 0$.
We can calculate the value $J(w)$ of our optimizer explicitly.
\begin{lemma}\label{lemma: optimal J value}
	For the optimal $w$ from \cref{lemma: strict local maximizer w variational problem} we have
	\[
		J(w) = (\Phi M_1)(\boundary(\assetsProcess_0)) = (\Phi M_1)(w(R)).
	\]
\end{lemma}
\begin{proof}
	By direct calculation we have $\squeeze{ {f M_1'}/\paren{\Phi M_2^2} \!=\! \paren{\paren{f \Phi' - f'\Phi}/\paren{f'\Phi'-f\Phi''}}' }$.
	Moreover, \eqref{eq: simplified Euler Lagrange} gives $e^r = f(y_0) / (\Phi M_2)(w(r))$.
	With $r= \transformThetaToR(\ell)$ and using \eqref{eq:boundary}, we get from \eqref{eq: expected proceeds with y and rho} that
	\begin{align*}
		J(w) &= e^{-\transformThetaToR(\assetsProcess_0)} \int_0^{\assetsProcess_0} f(\boundary(\ell)) e^{\transformThetaToR(\ell)} \diff \ell
	\\		&= (\Phi M_2)(\boundary(\assetsProcess_0)) \int_0^{\assetsProcess_0} \paren[\Big]{\frac{f}{\Phi M_2}}(\boundary(\ell)) \diff \ell
	\\		&= (\Phi M_2)(\boundary(\assetsProcess_0)) \int_{y_0}^{\boundary(\assetsProcess_0)} \paren[\Big]{ \frac{f M_1'}{\Phi M_2^2}}(x) \diff x
	\\		&= (\Phi M_2)(\boundary(\assetsProcess_0)) \brackets[\bigg]{ \frac{f \Phi' - f'\Phi}{f'\Phi'-f\Phi''} }_{y_0}^{\boundary(\assetsProcess_0)}
	\\		&= (\Phi M_1)(\boundary(\assetsProcess_0)). \qedhere
	\end{align*}
\end{proof}

Now we can translate the results obtained so far back to the state space of impact and asset position.
The following theorem will be crucial for our analysis in the verification arguments in \cref{sec:value fct}.
\begin{theorem}\label{thm: one-sided optimal y}
	The function $\boundary:[0,\infty) \to \RR$ defined by  \cref{eq:boundary} is a (one-sided) local maximizer of $\EE[\proceedsProcess_\infty\paren{\reflectionStrategy{\boundary}{\assetsProcess_0}} ]$ in the sense that, for every $\theta > 0$ there exists  $\varepsilon > 0$ such that for any decreasing $\tilde \boundary \in C^1( [0,\infty))$  
	with $\boundary(\cdot) \le \tilde \boundary(\cdot) \le y_0$,
	$\boundary = \tilde \boundary$ on $[\theta, \infty)$ 
	and $0 < \norm{\boundary-\tilde \boundary}_{W^{1,\infty}} < \varepsilon$, it holds 
	\[
		\EE\brackets[\big]{ \proceedsProcess_\infty\paren[\big]{\reflectionStrategy{\boundary}{\theta}} } > \EE\brackets[\big]{ \proceedsProcess_\infty\paren[\big]{\reflectionStrategy{\tilde \boundary}{\theta}} }.
	\]
\end{theorem}
\begin{proof}
	For sake of clarity, we write $J = J_R$ and $K = K_R$ to emphasize the dependence of the functionals $J,K$ on $R$.
	Call $w(r)$ the parametrization of $\boundary$ and $\tilde w(r)$ the parametrization of $\tilde \boundary$.
	
	Fix $\theta > 0$ and choose $R, \hat R, \hat \theta$ such that $\boundary(\theta) = w(R)$, $\tilde \boundary(\theta) = \tilde w(\hat R)$ and $w(\hat R) = \boundary(\hat \theta)$.
	So $R:= \transformThetaToR_\boundary(\theta)$, $\hat R:= \transformThetaToR_{\tilde \boundary}(\theta) = \int_0^\theta \frac{\Phi'}{\Phi}(\tilde \boundary(x)) \diff x + \int_{\tilde \boundary(\theta)}^{\tilde \boundary(0)} \frac{\Phi'}{\Phi}(u) \diff u$ and $\hat \theta:= \transformThetaToR_\boundary^{-1}(\hat R)$.
	By $\boundary \not\equiv \tilde \boundary$, $\boundary(\cdot) \le \tilde \boundary(\cdot)$ with equality outside $(0,\theta)$ and monotonicity of $\Phi'/\Phi$, we have $\hat R > R$ and thus $\hat \theta > \theta$.
	
	Now, $K_{\hat R}(w) = \hat \theta$ and $K_{\hat R}(\tilde w) = \theta$.
	Moreover, $J_r(w) = (\Phi M_1)(w(r))$ by \cref{lemma: optimal J value}.
	So if $\norm{w - \tilde w}_{W^{1,\infty}}$ is small enough, by \cref{lemma: strict local maximizer w variational problem} we get
	\begin{align*}
		\MoveEqLeft
		J_R(w) = \paren[\big]{\Phi M_1}\paren[\big]{w(R)} - \paren[\big]{\Phi M_1}\paren[\big]{w(\hat R)} + J_{\hat R}(w)
	\\
			&= \squeeze[2]{ \paren[\big]{\Phi M_1}\paren[\big]{w(R)} - \paren[\big]{\Phi M_1}\paren[\big]{w(\hat R)} + e^{-\hat R} f(y_0) \hat \theta + \paren[\big]{ J_{\hat R} - e^{-\hat R}f(y_0)  K_{\hat R} }(w) }
	\\
			&> \squeeze[2]{ \paren[\big]{\Phi M_1}\paren[\big]{w(R)} - \paren[\big]{\Phi M_1}\paren[\big]{w(\hat R)} + e^{-\hat R} f(y_0) \hat \theta + \paren[\big]{ J_{\hat R} - e^{-\hat R} f(y_0) K_{\hat R} }(\tilde w) }
	\\
			&= \paren[\big]{\Phi M_1}\paren[\big]{\boundary(\hat\theta - \eta)} - \paren[\big]{\Phi M_1}\paren[\big]{\boundary(\hat\theta)} + e^{-\hat R} f(y_0) \eta + J_{\hat R}(\tilde w) 
	\\
			&=: \Psi(\eta) + J_{\hat R}(\tilde w)\,.
	\end{align*}
	where  $\eta:= \hat \theta - \theta > 0$.
	By \eqref{eq: r via optimal w} we get $e^{-\hat R}f(y_0) = (\Phi M_2)(\boundary(\hat \theta))$.
	With~\eqref{eq: y' with M1 and M2}  follows
	\begin{align*}
		\Psi'(\eta) &= - \paren[\bigg]{(\Phi M_1)' \frac{M_2}{M_1'}}\paren[\big]{\boundary(\hat\theta - \eta)} + \paren[\big]{\Phi M_2}\paren[\big]{\boundary(\hat\theta) }
	\\
			&= - \paren[\bigg]{ \frac{\Phi' M_1 M_2}{M_1'} + \Phi M_2 }\paren[\big]{\boundary(\hat\theta - \eta)} + \paren[\big]{\Phi M_2}\paren[\big]{\boundary(\hat\theta)}.
	\end{align*}
	Hence $\Psi'(0) = - (\Phi' M_1 M_2 / M_1')(\boundary(\hat\theta))$. Since $M_1 > 0$ on  $(-\infty, y_0)$, $M_2 > 0$ on $(y_\infty, y_0]$, $M_1' < 0$ by \cref{lemma: M1 monotonicity} and $\Phi' > 0$, it follows $\Psi'(0) > 0$.
	So $\Psi(\eta) > 0$ for $\eta > 0$ small enough.
	Therefore we have by \eqref{problem: maximize with w}
	\[
		\EE\brackets[\big]{ \proceedsProcess_\infty\paren[\big]{\reflectionStrategy{\boundary}{\theta}} } 
			= J_R(w) > J_{\hat R}(\tilde w)
			= \EE\brackets[\big]{ \proceedsProcess_\infty\paren[\big]{\reflectionStrategy{\tilde \boundary}{\theta}} }.
	\]
	The bounds on $\eta$ and $\norm{w - \tilde w}_{W^{1,\infty}}$ are satisfied for small enough $\varepsilon > 0$, because $(\boundary,\ell) \mapsto \transformThetaToR_\boundary(\ell)$ and $(\boundary,\ell) \mapsto \transformThetaToR_\boundary^{-1}(\ell)$ are continuous in $W^{1,\infty} \times \RR$, so $\norm{w - \tilde w}_{W^{1,\infty}} \to 0$, $\hat R \to R$ and $\hat \theta \to \theta$ as $\varepsilon \to 0$.
\end{proof}

\section{Constructing the value function and verification}
\label{sec:value fct}

In this section, we construct a candidate for the value function and verify the variational inequality \eqref{eq:Variational} in \cref{lemma: wait inequality,lemma:inequality in sell region}, relying on results from the previous sections. This will be sufficient to conclude the proof of our main result, \cref{thm:optimal strategy}.

Having defined  a candidate boundary via the ODE \eqref{eq:boundary} to separate the sell and wait regions $\scS$ and $\scW$, we will now construct a solution  $V$ of the variational inequality \eqref{eq:Variational} that will give the value function of the optimal liquidation problem.
As a direct consequence of \cref{lemma: optimal J value}, we get its value along the boundary
\begin{equation}\label{eq: V on boundary}
	\bdryV(\theta):= V\paren[\big]{ \boundary(\theta),\theta } = \Phi\paren[\big]{ \boundary(\theta) } M_1\paren[\big]{ y(\theta) }\,.
\end{equation}
Inside the wait region $\scW$, which we assume is to the left of the boundary, we require $V=\waitV$ to satisfy $\frac{\impactVolatility^2}{2}V_{yy} + (\sigma \correlationBW \impactVolatility - \beta y)V_y = \delta V$.
Note that $\waitV$ solves the same ODE in $y$ as $\Phi$.
Since $V$ should be also monotonically increasing, the only possibility is that  $\waitV(y,\theta) = C(\theta)\Phi(y)$ for some increasing function $C:[0,\infty)\rightarrow [0,\infty)$.
Using the boundary condition $\waitV(\boundary(\theta),\theta) = \bdryV(\theta)$, in light of \cref{eq: V on boundary} we then have
\begin{equation}\label{def: V wait}
	\waitV(y,\theta) := \Phi(y) C(\theta)
\end{equation}
for $y \le \boundary(\theta)$ and $\theta \ge 0$, where $C(\theta) := M_1(\boundary(\theta))$.
On the other hand, in the sell region we require for $V = \bothSellV$ to satisfy $f = \bothSellV_y + \bothSellV_\theta$. We divide $\scS$ in two parts:
\begin{equation*}
\begin{split}
	\scS_1 &:= \{ (y,\theta) \in \RR \times (0,\infty) : \boundary(\theta) < y < y_0 + \theta \}\,,
\\	\scS_2 &:= \{ (y,\theta) \in \RR \times (0,\infty) : y_0 + \theta < y \}\,.
	\end{split}
\end{equation*}
Let $\Delta:= \Delta(y,\theta) \ge 0$ denote the $\norm{\cdot}_\infty$-distance of a point $(y,\theta) \in \overline\scS$ to the boundary $\partial \scS$ in direction $(-1,-1)$. This means in $\overline\scS_1$ (but not in $\scS_2$) that
\begin{equation}\label{eq:Delta and boundary}
	\boundary(\theta - \Delta) = y - \Delta \,.
\end{equation}
Inside $\overline\scS_1$, we need to have
\begin{equation} \label{def: V difficult sell region}
	\hardSellV(y,\theta) := \waitV(y - \Delta, \theta - \Delta) + \int_{y - \Delta}^y f(x) \diff x \,,
\end{equation}
since $\hardSellV_y + \hardSellV_\theta = f$ in $\overline\scS$ and $\hardSellV(\boundary(\theta),\theta) = \waitV(\boundary(\theta),\theta)$. Similarly, in $\overline{\scS}_2$,
\begin{equation}\label{eq:V in scS_2}
	V^{\scS_2}(y, \theta) := \int_{y-\theta}^y f(x)\diff x.
\end{equation}
To wrap up, the candidate value function is defined by:
\begin{equation}\label{eq: value fn constructed}
	V = V^\scW \text{ on }\overline\scW, \quad V = V^{\scS_{1}}  \text{ on } \overline\scS_{1}, \quad V = V^{\scS_{2}}  \text{ on }\overline\scS_{2}.
\end{equation}

The rest of this section is devoted to verifying that $V$  is a classical solution of the HJB variational inequality \eqref{eq:Variational} and thus concluding the proof of \cref{thm:optimal strategy} by an application of the martingale optimality principle. We first formalize the heuristic verification from \cref{sec:heuristics}.

\subsection{Martingale optimality principle}
\label{subsec: Martingale optimality principle}
Recall that $v$ is the value function of the optimal liquidation problem (cf.~\eqref{eq:def of value fn v}).
\begin{proposition}[Martingale optimality principle] \label{prop: supermatringale suffices}
	Consider a $C^{2,1}$ function $V:\RR \times [0,\infty) \to [0,\infty)$ with the following properties:
	\begin{enumerate}
		\item For every $\assetsProcess_{0-} \geq 0$, there exist constants $C_1, C_2$ so that
		\[
			V(y,\theta) \leq C_1 \exp(C_2 y)\vee 1 \quad \text{for all } (y,\theta)\in \RR\times [0, \assetsProcess_{0-}];
		\]
		\item For every $\assetsProcess_{0-}\geq 0$ and $A\in \monotoneStrategies{\assetsProcess_{0-}}$, the process $G$ from \eqref{eq:process G} is a supermartingale, where $Y = Y^{y,A}$ is defined in \eqref{def: impact dynamics}, and additionally $G_0(y;A) \le G_{0-}(y;A)$.
	\end{enumerate}	
	Then we have 
	\(
		\baseS_0 \cdot V(y, \theta) \geq v(y,\theta).
	\)\\
	Moreover, if there exists $A^{*} \in \monotoneStrategies{\assetsProcess_{0-}}$ such that $G(y;A^{*})$ is a martingale and  $G_0(y;A^*) = G_{0-}(y;A^*)$ holds, then we have
	\(
		\baseS_0 V(y, \theta) = v(y,\theta)
	\)
	and $v(y,\theta) = \EE[\proceedsProcess_\infty(y;A^*)]$ for $\assetsProcess_{0-} = \theta\ge 0$. In this case, any strategy $A$ for which $G(y;A)$ is not a martingale would be suboptimal.
\end{proposition}
\begin{proof}
By the supermartingale property  we have for every $T\geq 0$
\begin{align}
	\baseS_0 V(Y_{0-}, \assetsProcess_{0-}) &\geq \EE[G_0(y;A)] 
		\geq \EE[\proceedsProcess_T(y;A) + e^{-\gamma T}\baseS_T V(Y_T, \assetsProcess_T)] \nonumber\\
		& =  \EE[\proceedsProcess_T(y;A)] + e^{-\gamma T} \EE[\baseS_T V(Y_T, \assetsProcess_T)
\nonumber]
\\		
		&= \EE[\proceedsProcess_T(y;A)] + e^{-\delta T} \baseS_0 \EE[\scE(\sigma W)_T V(Y_T, \assetsProcess_T)] \label{eq:mart opt ineq}.
\end{align}
By monotone convergence, the first summand in \eqref{eq:mart opt ineq} tends to $\EE[\proceedsProcess_{\infty}(y;A)]$ for $T\rightarrow \infty$. 
To see that the second summand converges to $0$, consider the Ornstein-Uhlenbeck process
\(
	\diff X_t = -\beta X_t \diff t + \impactVolatility \diff B_t, \; X_0 = y.
\)
An application of It\^{o}'s formula 
gives
\begin{equation}\label{eq:Y and X}
	e^{\beta t} (Y_t - X_t) = \int_{[0,t]} e^{\beta u} \diff \assetsProcess_u \quad \forall t\geq 0.
\end{equation}
Since $\assetsProcess$ is non-increasing, we conclude $Y_t \leq X_t$ for all $t\geq 0$. Let $p,q>1$ be conjugate, i.e.~$1=1/q + 1/p$. Using H\"older's inequality and the bound on $V$,
\begin{align*}
	\MoveEqLeft \EE\brackets[\big]{ \scE(\sigma W)_T V(Y_T, \assetsProcess_T) } 
		\leq \EE\brackets[\big]{ \scE(\sigma W)_T^p }^{1/p} \EE\brackets[\big]{ V(Y_T, \assetsProcess_T)^q }^{1/q}
\\		&= \EE\brackets[\big]{ \exp\paren[\big]{ p \sigma W_T - \tfrac{1}{2} p \sigma^2 T} }^{1/p} \EE[V(Y_T, \assetsProcess_T)^q]^{1/q}
\\		&= \EE[\scE(p\sigma W)_T]^{1/p} \exp\paren[\Big]{ \tfrac{1}{p} \paren[\big]{\tfrac{1}{2} p^2 \sigma^2 T - \tfrac{1}{2} p \sigma^2 T } } \EE\brackets[\big]{ V(Y_T, \assetsProcess_T)^q }^{1/q}
\\		&= \exp\paren[\Big]{ \frac{p-1}{2}\sigma^2T } \EE[V(Y_T, \assetsProcess_T)^q]^{1/q} 
\\		&\leq \exp\paren[\Big]{ \frac{p-1}{2}\sigma^2T } \EE [ C_1^q \exp(q C_2 Y_T) \vee 1]^{1/q}
\\		&\leq \exp\paren[\Big]{ \frac{p-1}{2}\sigma^2T } { \EE [ C_1^q \exp(qC_2 X_T) \vee 1 ] }^{1/q}.
\end{align*}
Using the fact that $X$ is a Gaussian process with mean $\EE[X_T] = y e^{-\beta T}$ and variance $\Var(X_T) = \frac{\impactVolatility^2}{2\beta}(1 - e^{-2\beta T})$, we get for $K:={ \EE [ C_1^q \exp(qC_2 X_T) \vee 1 ]}$ that
\begin{align*}
	K &\leq 1 + C_1^q \exp\paren[\Big]{ q C_2 \EE[X_T] + \frac{1}{2} q^2 C_2^2 \Var(X_T) }
\\
		&\leq 1 + C_1^q  \exp\paren[\Big]{ q C_2 y + \frac{\impactVolatility^2}{4\beta} q^2 C_2^2  }.
\end{align*}
This bound on $K$ is independent of $T$.
Now choosing $p>1$ such that $\frac{p-1}{2} \sigma^2 < \delta$ ensures that $\exp(-\delta T) \exp\paren{ \frac{p-1}{2}\sigma^2T }$ is exponentially decreasing in $T$, and thus the second summand in \eqref{eq:mart opt ineq} converges to 0 for $T\rightarrow \infty$. This implies that
\(
	\baseS_0 V(y, \theta) \geq \EE[\proceedsProcess_\infty(y;A)]
\) for all $A \in \monotoneStrategies{\theta}$
and yields the first part of the claim.
The second part follows similarly by noting that, if $A^*\in  \monotoneStrategies{\theta}$ is such that $G(y;A^*)$ is a martingale and $G_0(y;A) = G_{0-}(y;A)$, then we have equalities instead of inequalities in the estimates leading to \eqref{eq:mart opt ineq}. By taking $T\rightarrow\infty$ we conclude that $\baseS_0 V(y, \theta) =  \EE[\proceedsProcess_\infty(y;A^*)]$. Since $\baseS_0 V(y, \theta) \geq v(y,\theta)$ by the first part of the claim, we deduce the optimality of $A^*$.
\end{proof}

To justify later why the stochastic integrals in \eqref{eq:Ito for G} are 
true martingales, we need  the following technical 

\begin{lemma} \label{lemma: true martingales}
Let $\assetsProcess_{0-} \geq 0$ be given and $F\in C^{2,1}(\RR \times [0,\infty) ; \RR)$ be such that there exist constants $C_1, C_2 \geq 0$ with  $\abs{ F(y, \theta) } \leq C_1 \exp (C_2 y) \vee 1$ for all $(y,\theta)\in \RR\times [0,\assetsProcess_{0-}]$.  For an admissible strategy $A\in \monotoneStrategies{\assetsProcess_{0-}}$ let $Y^A =: Y$ denote the impact process defined by \eqref{def: impact dynamics} for $y\in \RR$. Then the stochastic integral processes
\[
	 \int_0^{\cdot}\baseS_uF(Y_u, \assetsProcess_u)\diff B_u\quad \text{and}\quad  \int_0^{\cdot}\baseS_uF(Y_u, \assetsProcess_u)\diff W_u \quad \text{are true martingales.}
\]
\end{lemma}
\begin{proof}
It suffices to check $\EE[\int_0^t \baseS^2_u \exp (2 C_2 Y_u)\diff u] < \infty$ for every $t\geq 0$ by the exponential growth of $F$.
Consider the Ornstein-Uhlenbeck process $X$ given by
\(
	\diff X_t = -\beta X_t \diff t + \impactVolatility \diff B_t,
\) with $ X_0 = y$.
As in the proof of \cref{prop: supermatringale suffices} (see \eqref{eq:Y and X}), we have $Y_t \leq X_t$ for all $t\geq 0$. In particular, 
\begin{align*}
	&\EE\brackets[\Big]{ \int_0^t \baseS^2_u \exp (2 C_2 Y_u)\diff u }
		\leq \EE\brackets[\Big]{ \int_0^t \baseS^2_u \exp (2 C_2 X_u)\diff u }
\\		&= \int_0^t\EE[\baseS^2_u \exp (2 C_2 X_u)] \diff u 
		\leq \int_0^t\sqrt{\EE[\baseS^4_u] \EE[\exp (4 C_2 X_u)]}\diff u
		<\infty,
\end{align*}
using the Cauchy-Schwarz inequality  and the fact that $X$ is a Gaussian process.
\end{proof}

\subsection{Verification and proof of Theorem~\ref{thm:optimal strategy}}

Now we verify that $V$ is a classical solution of the variation inequality~\eqref{eq:Variational} with the  boundary condition $V(y, 0) = 0$ for all $y\in \RR$. That $V(y,0) = 0$ is clear because $M_1(y_0) = 0$. The rest will be split into several lemmas.

\begin{lemma}[Smooth pasting]\label{lemma: smooth paste}
	Let $(y_b, \theta_b)\in \overline\scW \cap \overline\scS$. Then 
	\begin{align}
		\label{eq:smooth pasting 1}
		\Phi(y_b)C'(\theta_b) + \Phi'(y_b)C(\theta_b) &= f(y_b)\,,
	\\
		\label{eq:smooth pasting 2}
		\Phi'(y_b)C'(\theta_b) + \Phi''(y_b)C(\theta_b) &= f'(y_b)\,. 
	\end{align}
\end{lemma}
\begin{proof}
	This follows easily from $C(\theta_b) = M_1(y_b)$ and $C'(\theta_b) = M_2(y_b)$, see the definition of $C$ and \eqref{eq:boundary}, together with the definitions of $M_1$ and $M_2$, see \eqref{def: M1 and M2}. 
	Note that when $(y_b,\theta_b) = (y_0, 0)$ we take the right derivative of $C$ at $0$ and the equalities still hold true.
\end{proof}
\begin{remark}\label{rmk: calculus of variations approach crucial}
It might be interesting to point out that \eqref{eq:smooth pasting 1} and \eqref{eq:smooth pasting 2} are sufficient to derive the boundary between the sell and the wait regions. Indeed, solving \eqref{eq:smooth pasting 1} -- \eqref{eq:smooth pasting 2} with respect to $C(\theta_b)$ and $C'(\theta_b)$, it is easy to see that $C(\theta_b) = M_1(y_b)$ and $C'(\theta_b) = M_2(y_b)$. On the other hand, by the chain rule one gets $\theta'(y_b) C'(\theta_b) = M_1'(y_b)$ and thus we derive for the boundary parametrization $\theta(\cdot) = \boundary^{-1}(\cdot)$ in the appropriate range
\[
	\theta'(y_b) = \frac{M_1'}{M_2}(y_b),
\]
which gives the ODE for the boundary in \eqref{eq:boundary}. To get the initial condition $y_0$, note that the boundary condition $V(\cdot,0) \equiv 0$ gives $C(0) = 0$, i.e.\ $M_1(y_0) = 0$, exactly as in \cref{lemma: boundary}. Thus, one could derive the candidate boundary function $\boundary(\cdot)$ after assuming sufficient smoothness of the function $V$ along the boundary. This is similar to the classical approach in the singular stochastic control literature, cf.~\cite[Section~6]{KaratzasShreve86}. The reason why we chose the seemingly longer derivation via calculus of variation techniques is the local (one-sided) optimality that we derived in \cref{thm: one-sided optimal y} and that will be crucial in our verification of the inequalities of the candidate value function in the sell region, see \cref{lemma:inequality in sell region}.
Even in the special case of $\lambda(\cdot)$ being constant,
 a more direct approach to verify the variational inequality is suggesting new, yet unproven (to our best knowledge), properties for quotients of Hermite functions that might be of independent interest, see \cref{rmk:conj about par cyl funs}.
\end{remark}

The smooth-pasting property translates to smoothness of $V$. Moreover, exponential bound on $V$ and $V_y$ will be needed to rely on the verification results from \cref{subsec: Martingale optimality principle}.

\begin{lemma}\label{lemma:V is smooth}
 The function $V$ is $C^{2,1}(\RR\times [0,\infty))$. Moreover, for every $\assetsProcess_{0-}$ there exist constants $C_1, C_2$, that depend on $\assetsProcess_{0-}$, such that both $V(y,\theta)$ and $V_y(y,\theta)$ are non-negative and bounded from above by $C_1 \exp(C_2 y)\vee 1$ for all $(y,\theta) \in \RR\times [0, \assetsProcess_{0-}]$.
\end{lemma}
\begin{proof}
 In $\scW$, the function $V$ is already $C^{2,1}$ by construction and the fact that $C(\theta) = M_1(\boundary(\theta))$ is continuously differentiable since $\boundary(\cdot)$ and $M_1(\cdot)$ are so.

 For $(y, \theta) \in \scS_1$, set $(y_b, \theta_b) := (y-\Delta(y,\theta), \theta-\Delta(y,\theta))$ and $\Delta := \Delta(y,\theta) $ (recall \eqref{eq:Delta and boundary}). 
 We have by \eqref{def: V difficult sell region} for the first and  \eqref{eq:smooth pasting 1} for the second equality 
\begin{align}
	\hardSellV_y 
		&= \Phi'(y_b)C(\theta_b) \, (1-\Delta_y) + \Phi(y_b)C'(\theta_b) \, (-\Delta_y) + f(y) - f(y_b) \, (1-\Delta_y) \notag
\\		
		&= \Phi'(y - \Delta)C(\theta - \Delta) + f(y) - f(y - \Delta). \label{eq:V_y in sell region}
\end{align}
Since $f$, $\Delta$, $C$ and $\Phi'$ are continuously differentiable, $V_y$ will also be so. 
Hence by  \eqref{eq:smooth pasting 2},
\begin{align}
	\hardSellV_{yy} 
		&=\Phi''(y_b)C(\theta_b) \, (1-\Delta_y) + \Phi'(y_b)C'(\theta_b) \, (-\Delta_y) + f'(y) - f'(y_b) \, (1 - \Delta_y)\notag
\\ 
		&= \waitV_{yy}(y_b,\theta_b) + f'(y) - f'(y_b),\label{eq:V_yy in sell region}
\end{align}
which is continuous.
On the other hand, by \eqref{def: V difficult sell region} and \eqref{eq:smooth pasting 2} we have
\begin{align}
	\hardSellV_\theta(y, \theta) 
		&= \Phi'(y_b)C(\theta_b)(-\Delta_\theta) + \Phi(y_b)C'(\theta_b)(1-\Delta_\theta) -f(y_b)(-\Delta_\theta)\notag
\\
		&= \Phi(y_b)C'(\theta_b), \label{eq:V_theta in sell region}
\end{align}
which is continuous.
For  $(y, \theta)\in \overline\scW\cap \overline\scS$ on the boundary,  
the left derivative w.r.t.~$y$ is 
\[
	\lim_{x\searrow 0} \frac 1 x \big({V(y,\theta) - V(y-x, \theta)}\big) = \Phi(y)C(\theta),
\]
while the right derivative is again given by \eqref{eq:V_y in sell region} and is equal to the left derivative since $\Delta(y,\theta) = 0$ in this case. 
Hence, $V$ is continuously differentiable w.r.t.~$y$ on the boundary with derivative $V_y(y, \theta) = \Phi'(y)C(\theta)$. 
Similarly, the left derivative of $V_y$ on the boundary is $\Phi''(y)C(\theta)$ and is equal to the right derivative which is given by  \eqref{eq:V_yy in sell region} with $y = y_b$. 
The left derivative of $V$ w.r.t.~$\theta$ on the boundary is equal to the right derivative (given by \eqref{eq:V_theta in sell region}). 
Therefore, $V$ is $C^{2,1}$ inside $\overline\scW\cup \scS_1$.
 
For $(y,\theta) \in \scS_2$, we have that $V^{\scS_2}_y = f(y) - f(y-\theta)$, $V^{\scS_2}_{yy} = f'(y) - f'(y-\theta)$ and $V^{\scS_2}_\theta =  f(y-\theta)$ by \eqref{eq:V in scS_2}, which are all continuous.
On the boundary between $\scS_1$ and $\scS_2$, the left derivative of $V$ w.r.t.~$y$ is given by $\eqref{eq:V_y in sell region}$ while the right derivative is $f(y) - f(y_0)$. Since $\theta-\Delta = 0$ in this case and $C(0) = 0$, they are equal and hence $V$ is continuously differentiable w.r.t.~$y$ there; similarly for $V_{yy}$. The left derivative of $V$ w.r.t.~$\theta$ there is given by \eqref{eq:V_theta in sell region} with $(y_b, \theta_b) = (y_0, 0)$. The right derivative w.r.t.~$\theta$ is $f(y-\theta) = f(y_0)$. They are equal by \eqref{eq:smooth pasting 2} and $C(0) = 0$. Therefore, $V$ is $C^{2,1}$ on $\overline\scS_1\cup \scS_2$. 
It remains to check smoothness on $\{(y,0) : y\in \RR\}$. The derivatives w.r.t.~$y$ there are 0. $V$ is continuously differentiable w.r.t.~$\theta$ in this case because $\boundary(\cdot)$, $C$,  and $\Delta$ are continuously differentiable w.r.t.~$\theta$ also at $\theta = 0$ (we consider the right derivatives in this case).  

To conclude the proof, the bound of $V$ and $V_y$ can be argued as follows. 
In the wait region, which is contained in $(-\infty, y_0]\times [0, \infty)$, we have $V(y, \theta) = C(\theta)\Phi(y)$ and  $V_y(y, \theta) = C(\theta)\Phi'(y)$. Since $\Phi,\Phi'$ are strictly increasing in $y$ (see \eqref{eq:Phi via H} and \cite[Chapter~10]{Lebedev1972special} for properties of the Hermite functions), $V$ and $V_y$ will be bounded by a constant there. 
Now, in the sell region we have $f - V_y - V_\theta = 0$. However,  $V_\theta > 0$ because in $\scS_1$ \eqref{eq:V_theta in sell region} holds and $C'(\theta_b) = M_2(\boundary(\theta_b)) > 0$, while in $\scS_2$ we have that $V_\theta(y,\theta) = f(y-\theta) > 0$. 
Similarly, $V_y > 0$ in the sell region.
Therefore, $0< V_y(y, \theta) < f(y) \leq \exp(\lambda_\infty y) \vee 1$ by \cref{ass:lambda is bounded}. 
Hence, integrating in $y$ gives $V(y,\theta) \le V(0,\theta) + \exp(\lambda_\infty y) / \lambda_\infty$ for $y \ge 0$, which implies $V(y,\theta) \le C_1 \exp(C_2 y) \vee 1$ for appropriate constants $C_1, C_2$.
\end{proof}

Next we prove that $V$ solves the variational inequality \eqref{eq:Variational}.
\begin{lemma}\label{lemma: wait inequality}
	The function $\waitV: \overline\scW \to [0,\infty)$ from \eqref{def: V wait} satisfies
	\[
		\scL\waitV(y,\theta) = 0%
\quad \text{and} \quad%
		f(y) < \waitV_y(y,\theta) + \waitV_\theta(y,\theta)
\text{ for $y < \boundary(\theta)$. }
	\]
\end{lemma}
\begin{proof}
	By \eqref{eq: y' with M1 and M2}, we have $\waitV_\theta = \Phi(y) M_1'(\boundary(\theta)) \boundary'(\theta) = \Phi(y) M_2(\boundary(\theta))$ and $\waitV_y = \Phi'(y) M_1(\boundary(\theta))$.
	Recall that at $y=\boundary(\theta)$ we have by \eqref{eq:smooth pasting 1} the equality $\waitV_y + \waitV_\theta = f(\boundary(\theta))$.
	Now consider $y<\boundary(\theta)$. By \cref{lemma: M1 monotonicity}, we then have $M_1(y) > M_1(\boundary(\theta))$ giving
	\[
		\paren[\Big]{\frac{f}{\Phi}}'(y) 
		> \paren[\Big]{\frac{\Phi'}{\Phi}}'(y) M_1\paren[\big]{ \boundary(\theta) } 
			= \frac{\diff}{\diff y} \paren[\bigg]{ M_1\paren[\big]{ \boundary(\theta) } \frac{\Phi'(y)}{\Phi(y)} + M_2\paren[\big]{ \boundary(\theta) } }.
	\]
	Therefore, $y\mapsto \paren{ f - \waitV_y(y,\theta) + \waitV_\theta(y,\theta) }/\Phi(y) $ is increasing in $y$.  
	Since at $y=\boundary(\theta)$ it equals to 0, we get the claimed inequality.
\end{proof}

It remains to verify the inequality in the sell region. The proof is more subtle and that is where \cref{thm: one-sided optimal y} plays a crucial role.
Recall \cref{ass:modelpara} and note that $y_\infty$ from \cref{lemma: boundary} is unique by
condition~\ref{ass:uniqueness of yinf}.
\begin{lemma}\label{lemma:inequality in sell region}
	The functions $\hardSellV$ and $\easySellV$ satisfy on $\overline\scS_1$ and $\scS_2$ respectively
	\[
		\scL \hardSellV \le 0, \quad \scL \easySellV < 0.
	\] 
	Moreover, the inequality inside $\overline\scS_1$ is strict except on the boundary between the wait region and the sell region ($\overline \scW\cap \overline \scS_1$) where we have equality.
\end{lemma}
\begin{proof}
First consider region $\overline\scS_1$. Recall from
 \cref{lemma:V is smooth} 
(see \eqref{eq:V_y in sell region} -- \eqref{eq:V_yy in sell region}) that in this case
\begin{align}
	\hardSellV_y(y,\theta) &= \waitV_y(y - \Delta, \theta - \Delta) + f(y) - f(y - \Delta),\notag
\\	\hardSellV_{yy}(y,\theta) &= \waitV_{yy}(y_b,\theta_b) + f'(y) - f'(y_b),\notag
\end{align}
where $y = y_b + \Delta(y,\theta)$ and $\theta = \theta_b + \Delta(y, \theta)$.
Fix $(y_b,\theta_b)\in \overline\scW\cap \overline\scS_1$ and consider  the perturbation $\Delta \mapsto (y,\theta) = (y_b + \Delta, \theta_b + \Delta)$.
Set
\begin{align*}
	h(\Delta)& :={} \scL \hardSellV(y_b + \Delta, \theta_b + \Delta)
\\
\begin{split}
		= {}&\tfrac{\impactVolatility^2}{2} \waitV_{yy}(y_b, \theta_b) 
			- \tfrac{\impactVolatility^2}{2} f'(y_b) 
			+ \sigma \correlationBW \impactVolatility \waitV_y(y_b, \theta_b) 
			- \sigma \correlationBW \impactVolatility f(y_b)
			- \delta \waitV(y_b, \theta_b)
\\		&+ \tfrac{\impactVolatility^2}{2} f'(y) 
			- \beta y \waitV_y(y_b, \theta_b) 
			+ \beta y f(y_b)
			+ (\sigma \correlationBW \impactVolatility - \beta y) f(y)
			- \delta \int_{y_b}^y f(x) \diff x \,.
\end{split}
\end{align*}
Note that $h(0) = 0$ by \cref{lemma: wait inequality} and to show $h(\Delta) < 0$ for $\Delta > 0$, it suffices to prove $h'(\Delta) < 0$ for all $\Delta > 0$.
We have for all $\Delta \geq 0$ at $y=y_b + \Delta$ that
\begin{align*}
\begin{split}
	h'(\Delta) &= \beta \paren[\big]{ f(y_b) - \waitV_y(y_b,\theta_b) } 
\\	&\qquad
		+ f(y) \paren[\bigg]{ \underbrace{ \frac{\impactVolatility^2}{2} \frac{f''(y)}{f(y)} - (\beta + \delta) + (\sigma \correlationBW \impactVolatility - \beta y) \frac{f'(y)}{f(y)} }_{= k(y)} } ,
\end{split}
\end{align*}
where at $\Delta = 0$ we consider the right derivative $h'(0+)$.
Now we show that $k(y) < 0$ for all $y\geq y_{\infty}$. 
To this end, recall that $\Phi$ is a solution of the ODE $\delta \Phi(x) = \frac{\impactVolatility^2}{2} \Phi''(x) + (\sigma \correlationBW \impactVolatility - \beta x) \Phi'(x)$.
Differentiating w.r.t.~$x$ and dividing by $\Phi'(x)$ yields
\begin{equation*}
	0 = \frac{\impactVolatility^2}{2} \paren*{ \frac{\Phi''(x)}{\Phi'(x)} }' + \frac{\impactVolatility^2}{2} \frac{\Phi''(x)^2}{\Phi'(x)^2} - (\beta + \delta) + (\sigma \correlationBW \impactVolatility - \beta x) \frac{\Phi''(x)}{\Phi'(x)}
\end{equation*}
So at the left end $y_\infty$ of our boundary, we have
\begin{align}
	k(y_\infty) &= \frac{\impactVolatility^2}{2} \paren[\bigg]{\frac{f'}{f}}'(y_\infty) + \frac{\impactVolatility^2}{2} \frac{\Phi''(y_\infty)^2}{\Phi'(y_\infty)^2} - (\beta + \delta) + (\sigma \correlationBW \impactVolatility - \beta y_\infty) \frac{\Phi''(y_\infty)}{\Phi'(y_\infty)} \notag
\\ \label{ineq:negative k in sell region}
		&= \frac{\impactVolatility^2}{2} \paren[\bigg]{\frac{f'}{f}}'(y_\infty) - \frac{\impactVolatility^2}{2} \paren[\bigg]{\frac{\Phi''}{\Phi'}}'(y_\infty) < 0
\end{align}
by \cref{ass:uniqueness of yinf}.
With \cref{ass:decreasing k} we get $k(y)< 0$ for every $y\geq y_\infty$.

In particular, $k(y_b + \Delta) < 0$ for all $\Delta \ge 0$. 
Since $f$ is positive and increasing, the product $\Delta \mapsto (fk)(y_b + \Delta)$ is decreasing.
Therefore, proving $h'(0+) \le 0$ is sufficient to show the inequality in $\scS_1$.
To stress the dependence of $h$ on the point $(y_b,\theta_b) = (\boundary(\theta_b),\theta_b)$, we also write $h(\Delta) = h_{\theta_b}(\Delta)$.
Note that $h_\theta(\Delta)$ is continuous in $\theta$ and $\Delta$ on $[0,\infty)\times [0,\infty)$.

Assume $h_{\theta_b}'(0+) > 0$ at some boundary point $(y_b,\theta_b)$ with $\theta_b > 0$. 
By continuity of $h'$ on $\theta$ and $\Delta$ there exists some $\varepsilon > 0$ such that $\scL \hardSellV > 0$ on $U:= \overline\scS_1 \cap B_\varepsilon(y_b,\theta_b)$. 
This will lead to a contradiction to the fact that the candidate boundary is a (one-sided) strict local maximizer of our stochastic optimization problem with strategies described by the local times of reflected diffusions, see \cref{thm: one-sided optimal y}. 

Indeed, fix $\assetsProcess_0 > \theta_b + \varepsilon$ and consider a perturbation $\tilde \boundary(\cdot) \in C^1$ of the boundary $\boundary(\cdot)$ which satisfies the conditions of \cref{thm: one-sided optimal y} and $\boundary(\theta) < \tilde \boundary(\theta) \le y_0$ in $(\tilde \boundary(\theta), \theta) \in U$ and such that $\tilde \boundary$ and $\boundary$ coincide outside of $U$. 
For the corresponding reflection strategies $\tilde A:= \reflectionStrategy{\tilde \boundary}{\assetsProcess_0}$ and $A:= \reflectionStrategy{\boundary}{\assetsProcess_0}$ denote by $\tilde \assetsProcess_t := \assetsProcess_0 - \tilde A_t$ and $\assetsProcess_t := \assetsProcess_0 - A_t$ their asset position processes.
The liquidation times of $\tilde A$ and $A$ are $\tilde \tau:= \inf \{ t \ge 0 : \tilde A_t = \assetsProcess_0 \}$ and $\tau:= \inf \{ t \ge 0 : A_t = \assetsProcess_0 \}$, respectively.
By \cref{lemma: Laplace transform of inv loc time} (see also the discussion after \eqref{LaplacetrafoTimeToLiquid}), we have $T:= \tilde\tau \vee \tau < \infty$~a.s.
Fix initial impact $Y^{\tilde A}_{0-} = Y^A_{0-} = \boundary(\assetsProcess_0)$. 
To compare the strategies $A$ and $\tilde A$, consider the processes $G(\boundary(\assetsProcess_0); A)$ and $G(\boundary(\assetsProcess_0); \tilde A)$ from \eqref{eq:process G} for our candidate value function (which is $C^{2,1}$ by \cref{lemma:V is smooth}).
Since $V(\cdot, 0) = 0$, we have $\proceedsProcess_T(\tilde A) = G_T(\tilde A)$ and $\proceedsProcess_T(A) = G_T(A)$.
However, since $(Y^{\tilde A}, \tilde \assetsProcess)$ spends a positive amount of time in the region $\{ \scL V > 0 \}$ until time $T$ and always remains in the region  $\{ \scL V \ge 0 \}$, the perturbed strategy $\tilde A$ generates larger proceeds (in expectation) than~$A$. 

Indeed, by~\eqref{eq:Ito for G} applied for $G(\tilde A)$  and $G(A)$, using monotone convergence (twice) and arguments as in the proof of \cref{prop: supermatringale suffices} for the first equality (by \eqref{eq: expected proceeds with y and Phi} expected proceeds are bounded), and \cref{lemma: true martingales} for the stochastic integrals in the second line (noting the growth condition from \cref{lemma:V is smooth}), we get
\begin{align*}
	\EE[\proceedsProcess_\infty(\tilde A) &- \proceedsProcess_\infty(A)] = \lim_{n\rightarrow \infty}\EE[G_{n\wedge T }(\tilde A) - G_{n\wedge T}(A)] 
\\		&=  \lim_{n\rightarrow \infty} \EE\brackets[\Big]{ \int_0^{n\wedge T} \dots \diff W_t + \int_0^{n\wedge T} \dots \diff B_t + \int_0^{n\wedge T} \scL V(Y^{\tilde A}_t, \tilde \assetsProcess_t) \diff t } 
\\		&=  \EE\brackets[\Big]{\int_0^{T} \scL V(Y^{\tilde A}_t, \tilde \assetsProcess_t) \diff t } > 0\,.
\end{align*}
This contradicts \cref{thm: one-sided optimal y}, so $h'(0+) \le 0$ and hence the inequality in $\scS_1$ must hold.

It remains to consider the case $(y,\theta)\in \overline\scS_2$, where $\easySellV_y  = f(y) - f(y-\theta)$ and $\easySellV_{yy} = f'(y) - f'(y-\theta)$. Fix $y-\theta =: a \geq y_0$ and consider $\scL \easySellV$ as a function of $\theta$. We  have
\begin{align*}
	\scL \easySellV(y,\theta) 
		&= \frac{\impactVolatility^2}{2} \paren[\big]{f'(a+\theta) - f'(a)} 
			+ \paren[\big]{\sigma \correlationBW \impactVolatility - \beta(a+\theta)} \paren[\big]{f(a+\theta) - f(a)} 
	\\	&\qquad
			- \delta \int_a^{a+\theta} \!\! f(x) \diff x
	.
\end{align*}
Differentiating the right-hand side w.r.t.~$\theta$ we get
\(
	f(a+\theta) k(a+\theta),
\)
which is again decreasing in $\theta$ because $a \geq y_0$. Since at $\theta = 0$ we have $\scL \easySellV(y, \theta) = 0$ we deduce the desired inequality.
\end{proof}

\begin{remark}\label{rmk:conj about par cyl funs}
In the particular case when $\lambda=f'/f$ is constant, a more direct approach based on straightforward calculations leads to a conjecture on a property for quotients of Hermite functions. 
More precisely, to prove $h'(0+)\leq 0$ in this case it turns out to be sufficient to verify that the map $y_b\mapsto h'(0)$ is monotone in $[y_\infty, y_0]$, because at $y_\infty$ and $y_0$ one can check that $h'(0+) < 0$. 
The monotonicity in $y_b$ would then follow from the following conjectured property of the Hermite functions:
\[
	\text{For every $\nu < 0$, the function }\quad x\mapsto \frac{(H_{\nu-1}(x))^2}{H_{\nu}(x) H_{\nu - 2}(x)} \quad \text{ is decreasing}.
\]
Numerical computations indicate the validity of the this property but, to our best  knowledge, it is not yet proven and may be of independent interest. Note that such quotients of special functions are related to so called Turan-type inequalities, cf.\ \cite{BariczIsmail2013}.
\end{remark}

Now we have all the ingredients in place to complete the

\begin{proof}[Proof of \cref{thm:optimal strategy}.] \label{ProofofMainthm}
The function $V$ constructed in \eqref{eq: value fn constructed} is a classical solution of the variational inequality \eqref{eq:Variational} because of \cref{lemma:V is smooth,,lemma: wait inequality,,lemma:inequality in sell region}.
Thus, for each admissible strategy $A$ the process $G(y;A)$ from \eqref{eq:process G} is a supermartingale with $G_{0}(y;A)\leq G_{0-}(y;A)$: the growth condition on $V_y$ and $V$ from \cref{lemma:V is smooth} guarantees that the stochastic integral processes in \eqref{eq:Ito for G} are true martingales by an application of \cref{lemma: true martingales}, while the variational inequality gives the supermartingale property on $[0-, \infty)$.
Moreover, for the described strategy $A^*$, whose existence and uniqueness on $\closedStochasticInterval{0, \tau}$ follows from classical results, cf.~\cref{rmk: existence+uniqueness of strong RSDE solution}, the process $G(y;A^*)$ is a true martingale with $G_{0}(y;A^*) = G_{0-}(y;A^*)$ by our construction of $V$ and the validity of the variational inequality in the respective regions.
Therefore $A^*$ is an optimal strategy by \cref{prop: supermatringale suffices}.
Any other strategy will be suboptimal because the respective inequalities are strict in the sell and wait region, i.e., for any other strategy the process $G$ will be a strict supermartingale.

The Laplace transform formula \eqref{LaplacetrafoTimeToLiquid} was derived in \cref{lemma: Laplace transform of inv loc time} for a $\boundary$-reflected strategy when the state process starts on the boundary. 
If the state process starts in $Y_0=x$ in the wait region, the behavior of the process until time $\hittingTimeFromTo{x}{z}$ when it hits the boundary for the first time (at $z:=\boundary(\assetsProcess_0)$) is independent from future excursions from the boundary, and hence the multiplicative factor in \eqref{LaplacetrafoTimeToLiquid}, 
see e.g.~\cite[Prop.~V.50.3]{RogersWilliams87vol2}:
for $x < z \in \RR$ and $\alpha > 0$,
\(
	\EE\brackets{ \exp\paren{-\alpha \hittingTimeFromTo{x}{z}} } = \Phi_\alpha(x) / \Phi_\alpha(z)
	.
\)
\end{proof}


\begin{thebibliography}{BBF17b}

\bibitem[AFS10]{AlfonsiFruthSchied10}
Aur{\'e}lien Alfonsi, Antje Fruth, and Alexander Schied.
\newblock Optimal execution strategies in limit order books with general shape
  functions.
\newblock {\em Quant. Finance}, 10(2):143--157, 2010.

\bibitem[AKS16]{AlfonsiKloeckSchied16}
Aurélien Alfonsi, Florian Klöck, and Alexander Schied.
\newblock Multivariate transient price impact and matrix-valued positive
  definite functions.
\newblock {\em Math. Oper. Res.}, 41(3):914--934, 2016.

\bibitem[Alm12]{Almgren12}
Robert Almgren.
\newblock Optimal trading with stochastic liquidity and volatility.
\newblock {\em SIAM J. Financial Math.}, 3(1):163--181, 2012.

\bibitem[AW95]{AbateWhitt1995}
Joseph Abate and Ward Whitt.
\newblock Numerical inversion of {L}aplace transforms of probability
  distributions.
\newblock {\em ORSA J. Comput.}, 7(1):36--43, 1995.

\bibitem[BBF17a]{BechererBilarevFrentrup2016-deterministic-liquidation}
Dirk Becherer, Todor Bilarev, and Peter Frentrup.
\newblock Optimal asset liquidation with multiplicative transient price impact.
\newblock {\em Appl. Math. Optim.}, 2017.

\bibitem[BBF17b]{BechererBilarevFrentrup-model-properties}
Dirk Becherer, Todor Bilarev, and Peter Frentrup.
\newblock Stability for gains from large investors' strategies in {M1}/{J1}
  topologies.
\newblock 2017.
\newblock arXiv:1701.02167v1.

\bibitem[BI13]{BariczIsmail2013}
{\'A}rp{\'a}d Baricz and Mourad E.~H. Ismail.
\newblock Tur\'an type inequalities for {T}ricomi confluent hypergeometric
  functions.
\newblock {\em Constr. Approx.}, 37(2):195--221, 2013.

\bibitem[CL95]{ChanL95}
Louis K.~C. Chan and Josef Lakonishok.
\newblock The behavior of stock prices around institutional trades.
\newblock {\em J. Finance}, 50(4):1147--1174, 1995.

\bibitem[DI93]{DupuisIshii93}
Paul Dupuis and Hitoshi Ishii.
\newblock {SDE}s with oblique reflection on nonsmooth domains.
\newblock {\em Ann. Probab.}, 21(1):554--580, 1993.

\bibitem[DM82]{DellacherieMeyer82bookB}
Claude Dellacherie and Paul-Andr{\'e} Meyer.
\newblock {\em Probabilities and Potential~{B}: Theory of Martingales}.
\newblock North-Holland, Amsterdam, 1982.

\bibitem[Fre98]{Frey1998}
R{\"u}diger Frey.
\newblock Perfect option hedging for a large trader.
\newblock {\em Finance Stoch.}, 2(2):115--141, 1998.

\bibitem[FSU17]{FruthSchonUrusov17}
Antje Fruth, Torsten Sch\"oneborn, and Mikhail Urusov.
\newblock Optimal trade execution in order books with stochastic liquidity.
\newblock {\em Preprint, available at
  \url{http://homepage.alice.de/murusov/papers/fsu-optimal_execution_stochastic.pdf}},
  2017.

\bibitem[GF00]{GelfandFomin00}
Israel~M. Gelfand and Sergei~V. Fomin.
\newblock {\em {Calculus of Variations}}.
\newblock Dover Books on Mathematics. Dover Publications, 2000.

\bibitem[GH16]{HorstGraewe16}
Paulwin Graewe and Ulrich Horst.
\newblock Optimal trade execution with instantaneous price impact and
  stochastic resilience.
\newblock {\em SIAM J. Control Optim.}, 2016.
\newblock to app.

\bibitem[GHS16]{GraeweHorstSere13}
Paulwin Graewe, Ulrich Horst, and Eric Séré.
\newblock Smooth solutions to portfolio liquidation problems under
  price-sensitive market impact.
\newblock {\em Stoch. Process. Appl.}, 2016.

\bibitem[JS03]{JacodShiryaev2003_book}
Jean Jacod and Albert~N. Shiryaev.
\newblock {\em Limit Theorems for Stochastic Processes}.
\newblock Springer, Berlin, second edition, 2003.

\bibitem[Kle08]{Klenke2008probability}
Achim Klenke.
\newblock {\em Probability Theory: A Comprehensive Course}.
\newblock Universitext. Springer London, 2008.

\bibitem[KS86]{KaratzasShreve86}
Ioannis Karatzas and Steven~E. Shreve.
\newblock Equivalent models for finite-fuel stochastic control.
\newblock {\em Stochastics}, 18(3-4):245--276, 1986.

\bibitem[Kyl85]{Kyle85}
Albert~S. Kyle.
\newblock Continuous auctions and insider trading.
\newblock {\em Econometrica}, 53(6):1315--1335, 1985.

\bibitem[Leb72]{Lebedev1972special}
Nikola{\u\i}~N. Lebedev.
\newblock {\em Special Functions and Their Applications}.
\newblock Dover Books on Mathematics. Dover Publications, 1972.

\bibitem[LN17]{LehalleNeuman17}
Charles-Albert Lehalle and Eyal Neuman.
\newblock Incorporating signals into optimal trading.
\newblock 2017.
\newblock arXiv:1704.00847v2.

\bibitem[LS13]{LorenzSchied13}
Christopher Lorenz and Alexander Schied.
\newblock Drift dependence of optimal trade execution strategies under
  transient price impact.
\newblock {\em Finance Stoch.}, 17(4):743--770, 2013.

\bibitem[PSS11]{PredoiuShaikhetShreve11}
Silviu Predoiu, Gennady Shaikhet, and Steven Shreve.
\newblock Optimal execution in a general one-sided limit-order book.
\newblock {\em SIAM J. Financial Math.}, 2(1):183--212, 2011.

\bibitem[PTW07]{PangTalrejaWhitt2007}
Guodong Pang, Rishi Talreja, and Ward Whitt.
\newblock Martingale proofs of many-server heavy-traffic limits for {M}arkovian
  queues.
\newblock {\em Probab. Surveys}, 4:193--267, 2007.

\bibitem[RW87]{RogersWilliams87vol2}
L.~Chris~G. Rogers and David Williams.
\newblock {\em {Diffusions, Markov Processes and Martingales, Vol. II: Itô
  Calculus}}.
\newblock Cambridge University Press, second edition, 1987.

\bibitem[SZ17]{SchiedZhang15}
Alexander Schied and Tao Zhang.
\newblock A state-constrained differential game arising in optimal portfolio
  liquidation.
\newblock {\em Math. Finance}, 27(3):779--802, 2017.

\end{thebibliography}
\end{document}